\documentclass[12pt]{amsart}
\usepackage{graphicx}
\usepackage{amssymb}
\usepackage{epstopdf}
\usepackage{color}
\usepackage{verbatim} 
\usepackage{array}
\usepackage{mathptmx}
\usepackage{bm}
\usepackage{latexsym}
\usepackage[mathscr]{euscript}

\newtheorem{theorem}{Theorem}[section]
\newtheorem{proposition}{Proposition}[section]
\newtheorem{lem}{Lemma}[section]
\newtheorem{corollary}{Corollary}[section]
\newtheorem{example}{Example}[section]
\newtheorem{rmk}{Remark}[section]
\newtheorem{definition}[theorem]{Definition}

\newcommand{\C}{\mathbb{C}}

\newcommand{\norm}[1]{\Vert#1\Vert }

\title[Hausdorff operators on Fock spaces]{Boundedness and compactness of Hausdorff operators on Fock spaces}
\author{Óscar Blasco, Antonio Galbis}

\subjclass[2020]{{30H20, 47B38, 47B10}}
\keywords{Hausdorff operator; Fock space; summing operator}
\thanks{The research of  \'O. Blasco was funded by the Spanish Ministry for Science and Innovation under Grant PID2022-138342NB-I00 for \emph{Functional analysis methods in approximation theory and applications.} The research of A. Galbis was partially supported by the projects MCIN PID2020-119457GBI00/AEI/10.13039/501100011033 and GV Prometeu/2021/070.}

\begin{document}

	\maketitle
\begin{abstract}
	We obtain a complete characterization of the bounded Hausdorff operators acting on a Fock space $F^p_\alpha$ and taking its values into a larger one $F^q_\alpha,\ 0 < p \leq q \leq \infty,$ as well as some necessary or sufficient conditions for a Hausdorff operator to transform a Fock space into a smaller one. Some results are written in the context of mixed norm Fock spaces. Also the compactness of Hausdorff operators on a Fock space is characterized. The compactness result for Hausdorff operators on the Fock space $F^\infty_\alpha$ is extended to more general Banach spaces of entire functions with weighted sup norms defined in terms of a radial weight and conditions for the Hausdorff operators to become  $p$-summing  are also included.
\end{abstract}

\section{Introduction}

The modern study of Hausdorff operators began with the work of Siskakis \cite{Siskakis} in complex analysis and the work of Liflyand-Móricz \cite{Liflyand-Moricz} in the Fourier transform setting. See for instance \cite{Chen,Kara, Liflyand} and the references therein. These operators have recently attracted the attention of many authors who have studied them in various spaces of analytic functions. For example, we refer to the articles \cite{Bonet, G-S, Hung, Mirotin, S}.  Following \cite{G-S}, for each  positive measure $\mu$ defined on $(0, \infty)$ we formally consider the {\it Hausdorff operator induced by the measure} $\mu$ defined by
$$
\mathscr{H}_\mu(f)(z):=\int_0^\infty f\left(\frac{z}{t}\right)\ \frac{d\mu(t)}{t},\ z\in {\mathbb C},$$ where $f\in H({\mathbb C})$ is an entire function and where we assume that $\int_{0}^\infty\frac{d\mu(t)}{t} < \infty$ for the operator to be defined on constant functions.
	
Our main goal is to study boundedness and compactness conditions for Hausdorff operators acting between different Fock spaces, although we also consider these operators on weighted Banach spaces of entire functions of type $H^\infty$ and analyze conditions for the operators to belong to certain classes of summing operators. Our results complement recent work by Galanopoulos and Stylogiannis \cite{G-S}, and by Bonet \cite{Bonet}.  In \cite[Theorem 1.1]{G-S} a necessary condition is obtained so that the Hausdorff operator acts continuously in a Fock space $F^p_\alpha,$ $1 \leq p \leq \infty,$ (see Section \ref{sec:pre} for the definition of Fock spaces and related classes). It is proved that if $\mu(0,1) = 0$ then $\mathscr{H}_\mu:F^p_\alpha\to F^p_\alpha$ is a bounded operator for $1\le p\le \infty$. As a consequence of our Theorem \ref{mainteo} it follows that condition $\mu(0,1) = 0$ holds if and only if $\mathscr{H}_\mu:F^p_\alpha\to F^p_\alpha$ is a bounded operator for every $0 < p < \infty$ if and only if $\mathscr{H}_\mu:F^p_\alpha\to F^q_\alpha$ is a bounded operator for some $0 < p \leq q\leq \infty.$ Our result also includes the case of Fock spaces $F^p_\alpha$ with $0 < p < 1,$ which are not Banach but quasi-Banach spaces and, since $F^p_\alpha\subset F^q_\alpha$ whenever $p<q$, gives that if $\mathscr{H}_\mu(F^p_\alpha)\subset F^q_\alpha$ for some $p<q$ then $\mathscr{H}_\mu(F^p_\alpha)\subset F^p_\alpha$ for all $0<p<\infty$.

On the other hand, in \cite{G-S} the authors also concentrate on $F^2_\alpha$ and they characterize the Hausdorff operators that belong to a Schatten class $S_p.$ Namely in \cite[Theorem 1.3]{G-S} it is shown that $\mathscr{H}_\mu$ belongs to the Schatten class $S_p$ for $1\le p<\infty$ if and only if $\sum_{n=0}^\infty \mu_n^p<\infty,$ where $\mu_n:= \int_0^\infty\frac{d\mu(t)}{t^{n+1}}.$
 Let us point out, denoting by $\Pi_{q,p}(X,Y)$, $\Pi_p(X,Y)$ and $N_p(X,Y)$ the classes of $(q,p)$-summming, $p$-summing and $p$-nuclear operators from $X$ into $Y$ respectively,  that for $X=Y=H$ a Hilbert space  it is known that  $$\Pi_p(H,H)=S_2, \quad 1\le p<\infty \quad \hbox{ (see \cite[Corollary 3.16]{Diestel})},$$
 $$N_1(H,H)= S_1, \quad N_p(H,H)=S_2, \quad 1< p<\infty\quad \hbox{ (see \cite[Theorem 5.30]{Diestel})}$$ and $$S_q=\Pi_{(q,2)}(H,H), \quad q\ge 2 \quad \hbox{(see \cite[Theorem 10.3]{Diestel} )}.$$ Hence we can rephrase the results in \cite{G-S} as $(\mu_n)\in \ell^1$ if and only if $\mathscr{H}_\mu\in N_1(F^2_\alpha,F^2_\alpha)$, $(\mu_n)\in \ell^2$ if and only
   $\mathscr{H}_\mu\in \Pi_{1}(F^2_\alpha,F^2_\alpha)$ or  $\mathscr{H}_\mu\in N_p(F^2_\alpha,F^2_\alpha)$ for some $1< p<\infty$  and  $(\mu_n)\in \ell^q$ for $q\ge 2$ if and only if $\mathscr{H}_\mu\in \Pi_{(q,2)}(F^2_\alpha,F^2_\alpha)$. We shall study then $(q,p)$-summing, $p$-summing and $p$-nuclear  Haussdorff operators acting on Fock spaces in the last section.

It should also be said that we not only consider Hausdorff operators on Fock spaces $F^p_\alpha$ but also on mixed normed Fock spaces (see Section \ref{sec:mixed}). To our knowledge, this is the first time these spaces are considered.

Besides the Introduction, the paper is divided into eight sections. Sections 2 and 3 are of preliminary character while Section 4 deals with boundedness of Hausdorff operators acting on Fock and mixed normed Fock spaces.
In Section \ref{sec:large-small} we obtain some necessary or sufficient conditions for a Hausdorff operator to transform a Fock space into a smaller one. We mention, for example, Theorem \ref{teopq}, which is based on an improvement of results about the Taylor coefficients of functions in the Fock space obtained in \cite{BG,T}, or Theorem \ref{th:fromF1toFinf1} which gives a complete characterization of those Hausdorff operators mapping $F^1_\alpha$ into the smaller mixed Fock space $F^{\infty,1,\alpha}.$

Sections 6 and 7 deal with compactness. Our second main result is Theorem \ref{th:compact} which states that $\mathscr{H}_\mu:F^p_\alpha\to F^q_\alpha$ is compact for some $ 0 < p \leq q\leq \infty$ if and only if $\mathscr{H}_\mu:F^p_\alpha\to F^p_\alpha$ is compact for every $0 < p \leq \infty$ and this happens precisely when $\mu\left((0,1]\right) = 0.$ This should be compared with \cite[Theorem 1.2]{G-S}, where only the case $1 < p=q < \infty$ was considered. We should mention that in \cite{Bonet} boundedness and compactness results are obtained for Hausdorff operators acting on weighted Banach spaces of type $H^\infty$ (see section \ref{sec:weighted} for the precise definition). In particular, a necessary condition is obtained in \cite{Bonet} for compactness as well as some sufficient conditions in concrete cases. We improve those results and in Theorem \ref{th:weighted} we extend the characterization of compactness for Hausdorff operators acting on the Fock space $F^\infty_\alpha$ to more general Banach spaces of entire functions with weighted sup norms. This is done by combining the techniques of the proof of theorem \ref{mainteo} with some ideas contained in \cite{Lusky}.
Section 8 contains concrete examples, in which some of the obtained results are applied to variants of the Hardy operator or certain fractional operators. Finally in the last section we study $p$-summing, $p$-nuclear  and $(q,p)$-summing operators on Fock spaces that are not Hilbert spaces and find a connection with the fact that the operator maps a Fock space into a smaller class.

\section{Preliminaries and notation}\label{sec:pre}

In this section we recall the definition of Fock spaces and their basic properties. As usual $dA$ denotes the area measure on the complex plane.

\begin{definition} \label{Z1}Let $0<p\le\infty$ and $\alpha>0$. The Fock space $F^p_\alpha$ consists of all entire functions $f$ such that $\|f\|_{p,\alpha} < \infty,$ where
	$$\|f\|_{p,\alpha}= \Big(\frac{\alpha p}{2\pi}\int_\C|f(z)e^{-\frac{\alpha}{2}|z|^2}|^pdA(z)\Big)^{1/p}$$ for $0 < p < \infty$ while
	$$\|f\|_{\infty,\alpha}=\sup_{z\in \C} |f(z)|e^{-\frac{\alpha}{2}|z|^2}.$$
We denote by $f^\infty_\alpha$ the subspace of $F^\infty_\alpha$ given by entire functions $f$ such that
$$\lim_{|z|\to \infty} |f(z)|e^{-\frac{\alpha}{2}|z|^2}=0.$$
\end{definition}  In the case $1 \leq p\leq \infty,$ $F^p_\alpha$ is a Banach space with norm $\|\cdot\|_{p,\alpha}.$ For $0 < p < 1,$ $F^p_\alpha$ is a complete metric space with the distance $d(f,g) = \|f-g\|_{p,\alpha}^p.$
\par\medskip
 It turns out that $F^p_\alpha\subset F^q_\alpha\subset f^\infty_\alpha$ with continuous inclusions for every $0 < p < q < \infty$ (see \cite[Theorem 2.10]{Z1}).
	\par\medskip
As usual $$M_p(f,r)=\left(\frac{1}{2\pi}\int_0^{2\pi} |f(re^{i\theta})|^p d\theta\right)^{1/p},\  0 < p < \infty$$ and $\displaystyle M_\infty(f,r)=\sup_{|z|=r}|f(z)|$. Hence
\begin{equation}\label{eq:norm-fock}\|f\|_{p,\alpha}= \Big(\alpha p\int_0^\infty M_p(f,r)^pe^{-\frac{\alpha p}{2}r^2} r\ dr\Big)^{1/p},\ 0 < p < \infty,\end{equation}
and
$$\|f\|_{\infty,\alpha}=\sup_{r\ge 0} M_\infty(f,r)e^{-\frac{\alpha}{2}r^2}.$$
In the sequel we will denote $u_n(z) = z^n.$ According to \cite[Page 40]{Z1}
\begin{equation}\label{nun}
	\|u_n\|^p_{p,\alpha}= (\frac{2}{\alpha p})^{\frac{np}{2}}\Gamma(\frac{np}{2}+1), \quad 0 < p < \infty,
\end{equation} and
\begin{equation}\label{nuinf}
	\|u_n\|_{\infty,\alpha}= (\frac{n}{\alpha e})^{n/2}.
\end{equation}
\par\medskip
In what follows we use the notation $A_n\lesssim B_n$ to mean the existence of $c > 0$ such that
$A_n \leq c B_n \ \forall n\in {\mathbb N}.$ We also set $A_n\asymp B_n$ when $A_n\lesssim B_n$ and $B_n\lesssim A_n.$
When we do not use this notation in proofs, a constant $C$ may appear that is not always the same.
\par\medskip
Then for $n\in \mathbb N,$
\begin{equation}\label{normaun}
\|u_n\|_{p,\alpha}\asymp (\frac{n}{\alpha e})^{n/2} n^{\frac{1}{2p}}\asymp \sqrt{\frac{n!}{\alpha^n}} n^{\frac{1}{2p}-\frac{1}{4}},\ \ \|u_n\|_{\infty,\alpha} \asymp \sqrt{\frac{n!}{\alpha^n}} n^{-\frac{1}{4}}.\end{equation}

It is known that the Taylor series of a function in $F^1_\alpha$ does not have to converge to the function in the norm of $F^1_\alpha$ . Despite this, we will sometimes use the notation $\displaystyle f = \sum_{n=0}^{\infty}a_n u_n$ to indicate that $\displaystyle f(z) = \sum_{n=0}^{\infty}a_n z^n$ for every $z\in {\mathbb C}.$ For $\displaystyle f = \sum_{n=0}^{\infty}a_n u_n\in F^2_\alpha$ we have
$$
\|f\|^2_{2,\alpha} = \sum_{n=0}^\infty |a_n|^2 \frac{n!}{\alpha^n}.$$
\par\medskip
We recall that the reproducing kernel of $F^2_\alpha$ is given by
$$
K_\alpha(z,w)=e^{\alpha z\bar w}= \sum_{n=0}^\infty \frac{\alpha^n \bar w}{n!} u_n(z).$$ Hence \begin{equation}\label{repro}f(z)= \int_\C K_\alpha(z,w) f(w) d\lambda_\alpha(w)\end{equation}
for any $f\in F^2_\alpha$ where  $d\lambda_\alpha(w)=\frac{\alpha}{\pi} e^{-\alpha |w|^2} dA(w)$ for any $\alpha>0$.
In particular
\begin{equation} \|K_\alpha(\cdot, a)\|_{2,\alpha}= K_\alpha(a, a)^{1/2}=e^{\frac{\alpha}{2}|a|^2},\ a\in \mathbb C.
\end{equation}

\begin{proposition} \label{p1} Let $\alpha,\beta>0$ and $a\in \C$. Then
\begin{equation} \label{nrk}
\|K_\beta(\cdot,a)\|_{p,\alpha}= e^{\frac{\beta^2 |a|^2}{2\alpha}}, \quad 0<p\le\infty
\end{equation}
\end{proposition}
\begin{proof} It was shown in \cite[Corollary 2.5]{Z1} that
\begin{equation} \label{norma1}
\|K_\beta(\cdot, a)\|_{L^1(\lambda_\alpha)}= e^{\frac{\beta^2|a|^2}{4\alpha}}, \quad \alpha,\beta>0,\ a\in \C.
\end{equation}

For the case $0 < p < \infty,$ taking into account that
$\|f\|^p_{p,\alpha}= \| |f|^p\|_{L^1(\lambda_{p\alpha/2})}$ we obtain that
$$\|K_\beta(\cdot,a)\|^p_{p,\alpha}= \|K_{p\beta}(\cdot,a)\|_{L^1(\lambda_{p\alpha/2})}= e^{\frac{\beta^2p|a|^2}{2\alpha}}.$$
On the other hand,
$$ \|K_\alpha(\cdot,a)\|_{\infty,\alpha}=\sup_{z\in \C} e^{-\frac{\alpha}{2} (|z|^2-2\Re(z\bar a))}=e^{\frac{\alpha}{2}|a|^2}.$$
Hence, since $\|K_\beta(\cdot,a)\|_{\infty,\alpha}=\|K_\alpha(\cdot, \frac{\beta}{\alpha} a)\|_{\infty,\alpha}$ we obtain the case $p=\infty$.
\end{proof}

To finish this section we recall that for $1\leq p < \infty,$ the dual $(F^p_\alpha)^\ast$ can be identified with $F^{p'}_\alpha,\ \frac{1}{p} + \frac{1}{p'} = 1,$ (see \cite[Theorem 2.23]{Z1}) and  $(f^\infty_\alpha)^\ast= F^1_\alpha$ (see \cite[Theorem 2.26]{Z1}) under the pairing
$$
\langle f, g\rangle_\alpha = \frac{\alpha}{\pi}\int_{{\mathbb C}}f(z) \overline{g(z)} e^{-\alpha |z|^2}\ d A(z),\ f\in F^p_\alpha,\ g\in F^{p'}_\alpha.$$ Moreover, for $\displaystyle f = \sum_{n=0}^{\infty}a_n u_n$ and $\displaystyle g = \sum_{n=0}^{\infty}b_n u_n$  we have $\displaystyle \langle f, g\rangle_\alpha = \sum_{n=0}^\infty a_n \overline{b_n}\frac{n!}{\alpha^n}.$

\section{Mixed normed Fock spaces}\label{sec:mixed}

The expression (\ref{eq:norm-fock}) for the norm in the Fock spaces suggests the following definition.

\begin{definition}\label{Z2} Let $0 < p, q\leq\infty$ and $\alpha>0.$ The Fock space $F^{p,q,\alpha}$ consists of all entire functions $f$ such that $\|f\|_{p,q,\alpha} < \infty,$ where
	$$\|f\|_{p,q,\alpha}= \Big(\alpha q\int_0^\infty M_p(f,r)^qe^{-\frac{\alpha q}{2}r^2} rdr\Big)^{1/q} < \infty,\ \quad 0< q < \infty,$$ and 	
	$$\|f\|_{p,\infty,\alpha} = \sup_{r\geq 0}\ M_p(f,r)e^{-\frac{\alpha}{2}r^2}.$$
\end{definition}

In the case $p = q$ we recover the Fock spaces: $F^{p,p,\alpha} = F^p_\alpha.$
\par\medskip
Since $M_p(u_n, r) = M_q(u_n, r) = r^n$ for all $0 < p,q\leq \infty$ then \begin{equation}\label{eq:mixednorm-monomial}\|u_n\|_{p,q,\alpha} = \|u_n\|_{q,\alpha}.\end{equation}
\par\medskip
The family $F^{p,q,\alpha}$ decreases with $p$ and increases with $q.$

\begin{proposition}\label{embeddings1} Let $0<p,q\le \infty$, $0<p_1\le p_2\le \infty$ and $0<q_1\le q_2\le \infty$. Then
\begin{itemize}
\item[(i)] $F^{\infty,q,\alpha}\subseteq F^{p_2,q,\alpha} \subseteq F^{p_1,q,\alpha}$ and $$ \|f\|_{p_1,q,\alpha}\le \|f\|_{p_2,q,\alpha}\le \|f\|_{\infty,q,\alpha}.$$
\item[(ii)] $F^{p,q_1,\alpha}\subseteq F^{p,q_2,\alpha} \subseteq F^{p,\infty,\alpha}$ and $$\|f\|_{p,\infty,\alpha}\le \|f\|_{p,q_2,\alpha}\le (\frac{q_2}{q_1})^{1/q_2}\|f\|_{p,q_1,\alpha}.$$
\end{itemize}	
\end{proposition}
\begin{proof}
Statement (i) follows from $M_{p_1}(f,r)\le M_{p_2}(f,r)\le M_\infty(f,r).$

To check (ii) we first take $f\in  F^{p,q_2,\alpha}$. Since $M_p(f,r)$ is increasing we have
$$ M^{q_2}_p(f,r) e^{-\frac{\alpha q_2}{2}r^2}\le (\alpha q_2)\int_r^\infty M_p(f,s)^{q_2}e^{-\frac{\alpha q_2}{2}s^2}sds\le\|f\|^{q_2}_{p,q_2,\alpha}. $$
This gives the inclusion $F^{p,q_2,\alpha} \subseteq F^{p,\infty,\alpha}$ and the inequality $ \|f\|_{p,\infty,\alpha}\le \|f\|_{p,q_2,\alpha}$.
On the other hand
if $f\in  F^{p,q_1,\alpha}$ and using that $F^{p,q_1,\alpha}\subset F^{p,\infty,\alpha}$ we have
$$
\begin{array}{*2{>{\displaystyle}l}} \|f\|^{q_2}_{p,q_2,\alpha}&\le \|f\|_{p,\infty,\alpha}^{q_2-q_1}\frac{q_2}{q_1}(\alpha q_1)\int_0^\infty M_p(f,r)^{q_1}e^{-\frac{\alpha q_1}{2}r^2} r\ dr\\ & \\
&= \frac{q_2}{q_1}\|f\|_{p,q_1,\alpha}^{q_1} \|f\|_{p,\infty,\alpha}^{q_2-q_1}\le \frac{q_2}{q_1}\|f\|_{p,q_1,\alpha}^{q_2}.
\end{array}$$
This shows the other inclusion and the inequality of norms.
\end{proof}

\begin{corollary} Let $0 < p < q < \infty.$ Then
	$$
	F^{\infty,p,\alpha}\subset F^{q,p,\alpha}\subset F^p_\alpha\subset F^q_\alpha\subset F^{p,q,\alpha}\subset F^{p,\infty,\alpha}.$$
\end{corollary}

\section{Hausdorff operators. Boundedness}\label{sec:bounded}

Let $\mu$ be a positive Borel measure on $(0, \infty)$. The Hausdorff operator is formally defined by
$$
\mathscr{H}_\mu(f)(z):=\int_0^\infty f\left(\frac{z}{t}\right)\ \frac{d\mu(t)}{t},\ \ f\in H({\mathbb C}).$$
Since we want $\mathscr{H}_\mu$ to be defined on $u_n$ for $n\in \mathbb N_0$ we need to assume that $\mu_n:= \int_0^\infty\frac{d\mu(t)}{t^{n+1}}<\infty,\ n\in {\mathbb N}_0.$ We first discuss the continuity of the Hausdorff operator when acting on the Fréchet space $H({\mathbb C})$ of entire functions. We recall that the compact open topology on that space can be described in terms of the family of norms $f\mapsto M_\infty(f,r),\ r > 0.$
Let us first present a lemma to be used later on.
\begin{lem}\label{prop:momentos} Let $\mu$ be a positive Borel measure on $(0, \infty)$ such that $\mu_n<\infty$ for $n\in \mathbb N_0$. Then
 $$\displaystyle\sup_n \sqrt[n]{\mu_n}<\infty \Longleftrightarrow \hbox{ there exists } \delta>0 \hbox{ such that }
		  \mu(0, \delta) = 0.$$
	Moreover $\delta$ can be taken $\frac{1}{\sup_{n\in \mathbb N} \sqrt[n]{\mu_n}}$ and $\sup_{n\in \mathbb N} \sqrt[n]{\mu_n}\le \frac{\max\{\mu_0,1\}}{\delta_0}$ where $\delta_0=\sup\{\delta>0: \mu(0,\delta)=0\}$.
\end{lem}
\begin{proof}
Assume that $C = \displaystyle\sup_n \sqrt[n]{\mu_n} < \infty$ and let $\delta<\frac{1}{C}$ be given. Then
	$$
	\frac{\mu(0,\delta)}{\delta^{n+1}} \leq \int_0^\delta \frac{d\mu(t)}{t^{n+1}} \leq C^n,$$ from where it follows $\mu(0,\delta)\leq \delta\left(\delta C\right)^n$ for every $n\in {\mathbb N}.$ Then $\mu(0,\delta) = 0$ for any $\delta<\frac{1}{C}$ and therefore $\mu(0, \frac{1}{C})=0$

	Let us now assume $\mu\neq 0$ and $\mu(0, \delta) = 0$ for some $\delta > 0$ and consider $\delta_0=\sup\{ \delta>0: \mu(0,\delta)=0\}$. Then
	$$
	\mu_n = \int_{\delta_0}^\infty \frac{d\mu(t)}{t^{n+1}} \leq \frac{1}{\delta_0^n}\int_{\delta_0}^\infty\frac{d\mu(t)}{t}\le \frac{\mu_0}{\delta_0^n}.$$
Hence $\sqrt[n]{\mu_n}\le \frac{\max\{\mu_0,1\}}{\delta_0}$ for all $n\in \mathbb N$.
\end{proof}

\begin{proposition}\label{prop:continuity-entire} The following conditions are equivalent:
	\begin{itemize}
		\item[(i)] $\mathscr{H}_\mu:H({\mathbb C})\to H({\mathbb C})$ is well-defined and continuous.
		\item[(ii)] $\displaystyle\sup_n \sqrt[n]{\mu_n} < \infty.$
	\end{itemize}
\end{proposition}
\begin{proof}
	$(i)\Rightarrow (ii).$ If $\mathscr{H}_\mu(f)$ is well-defined for $u_n$ then $\mu_n < \infty.$ From the continuity we have that there exist $C > 0$ and $r\geq 1$ such that
	$$
	M_\infty\left(\mathscr{H}_\mu(f), 1\right)\leq C M_\infty(f,r)\ \ \forall f\in H({\mathbb C}).$$ We now take $f = u_n$ to conclude $\mu_n \leq C r^n$ for every $n\in {\mathbb N}_0.$
	\par\medskip
	$(ii)\Rightarrow (i).$ For every entire function $\displaystyle f(z) = \sum_{n=0}^\infty a_n z^n$ we have
	$$
	\sum_{n=0}^\infty |a_n|\cdot |z|^n \int_0^\infty\frac{d\mu(t)}{t^{n+1}} =\sum_{n=0}^\infty |a_n|\mu_n |z|^n < \infty,$$ so convergence dominated theorem permits to conclude that
	$$
	\mathscr{H}_\mu(f)(z) = \sum_{n=0}^\infty a_n \mu_n z^n$$ and $\mathscr{H}_\mu(f)$ is an entire function. We now fix $R\geq 1$ so that $\mu_n \leq R^n$ for every $n\in {\mathbb N}_0.$ For every $r > 0$ and $\displaystyle f(z) = \sum_{n=0}^\infty a_n z^n$ we have $(2rR)^n |a_n| \leq M_\infty(f, 2rR).$ Hence
	$$
	M_\infty\left(\mathscr{H}_\mu(f), r\right)\leq \sum_{n=0}^\infty |a_n|\mu_n r^n \leq M_\infty(f, 2rR) \sum_{n=0}^\infty 2^{-n},$$ from where it follows the continuity of $\mathscr{H}_\mu:H({\mathbb C})\to H({\mathbb C}).$
	\par\medskip
\end{proof}
\par\medskip
We want to discuss under which conditions $\mathscr{H}_\mu$ defines a bounded operator on mixed Fock spaces.
\par\medskip
As mentioned above we have to impose the condition that $\mu_n < \infty$ for every $n\in {\mathbb N}_0$, since it is essential for $\mathscr{H}_\mu(u_n)$ to make sense, and there is no loss of generality if we assume that $\mu_0 = 1.$
Under the extra condition that $\mu(0, \delta) = 0$ for some $\delta > 0$, by Lemma \ref{prop:momentos} and Proposition \ref{prop:continuity-entire}, the operator $\mathscr{H}_\mu:F^{p_1, q_1, \alpha}\to F^{p_2, q_2, \alpha}$ has closed graph as long as it is well defined and then
the continuity of $\mathscr{H}_\mu:F^{p_1, q_1, \alpha}\to F^{p_2, q_2, \alpha}$ is equivalent to the inclusion $\mathscr{H}_\mu(F^{p_1, q_1, \alpha}) \subset F^{p_2, q_2, \alpha}.$

\begin{theorem}\label{mainteo} Let $\mu$ be a positive Borel measure on $(0,\infty)$ satisfying the condition $\int_0^\infty \frac{d\mu(t)}{t}=1.$ The following are equivalent:
	\begin{itemize}
		\item[(i)] $\mathscr{H}_\mu: F^p_\alpha\to F^p_\alpha$ is bounded for all $0<p\le \infty$ and $\alpha>0.$
		\item[(ii)] There exist $0<p\le \infty$ and $\alpha>0$ such that $\mathscr{H}_\mu: F^p_\alpha\to F^p_\alpha$ is bounded.
		\item[(iii)] There exist $0<p,q\leq \infty$ and $\alpha>0$ such that $\mathscr{H}_\mu:  F^{\infty,q,\alpha}\to F^{p,\infty,\alpha}$ is bounded.
		\item[(iv)] $\mu(0,1)=0.$
		\item[(v)] $\sup_{n\ge 0} \mu_n < \infty.$
		\item[(vi)] $\mathscr{H}_\mu: F^{p,q,\alpha}\to F^{p,q,\alpha}$ is bounded for all $0<p,q\le \infty$ and $\alpha>0.$
	\end{itemize}
\end{theorem}
\begin{proof} (i) $\Longrightarrow$ (ii) It is obvious.

	(ii) $\Longrightarrow$ (iii) It follows using that $F^{\infty,p,\alpha}\subset F^{p}_\alpha \subset F^{p,\infty,\alpha}$.
	
	(iii) $\Longrightarrow$ (iv)  We will assume that $0 < q < \infty$ since the case $q = \infty$ is easier. We have $\|\mathscr{H}_\mu(u_n)\|_{\infty,\alpha} \lesssim \|u_n\|_{q,\alpha}.$  Hence, using (\ref{normaun}), we obtain
	$$
	\mu_n \lesssim n^{\frac{1}{2q}},$$ that is,
	$$
	\sup_{n\in {\mathbb N}}\frac{1}{n^a}\int_0^\infty \frac{d\mu(t)}{t^{n+1}} \leq C < \infty,$$ where $a = \frac{1}{2q}.$ For every $0 < \delta < 1$ we have
	$$
	\left(\frac{1}{\delta}\right)^{n+1}\frac{\mu(0,\delta)}{n^a} \leq \frac{1}{n^a}\int_0^\delta \frac{d\mu(t)}{t^{n+1}} \leq C\ \ \forall n\in {\mathbb N},$$ from where it follows $\mu(0,\delta) = 0$ for any $\delta<1$ and we get the result.

	(iv) $\Longrightarrow$ (v) Condition (iv) implies that  $\displaystyle\mu_n = \int_1^\infty\frac{d\mu(t)}{t^{n+1}}$. Hence $\mu_n$ is decreasing and $\sup_{n\in \mathbb N_0}\mu_n=\mu_0.$
	
	(v) $\Longrightarrow$ (iv) Take $\phi(t) = \frac{1}{t}.$ Since $d\nu(t) = \frac{d\mu(t)}{t}$ is a probability measure and $\sqrt[n]{\mu_n}= \|\phi\|_{L^n(d\nu)}$ then $\left(\sqrt[n]{\mu_n}\right)_n $ is increasing. Condition (v) gives that $\sup_{n}\sqrt[n]{\mu_n}=\lim_n\sqrt[n]{\mu_n} \leq 1.$ Then Lemma \ref{prop:momentos} implies that $\mu(0,1) = 0.$
	
	(iv) $\Longrightarrow$ (vi) Since $$\mathscr{H}_\mu(f)(z)=\int_1^\infty f(\frac{z}{t})\frac{d\mu(t)}{t}$$ then
 $$M_p(\mathscr{H}_\mu(f),r)\le \int_1^\infty M_p(f,r/t))\frac{d\mu(t)}{t}, \quad 1\le p\le \infty$$
 and
 $$M^p_p(\mathscr{H}_\mu(f),r)\le \int_1^\infty M^p_p(f,r/t))\frac{d\mu(t)}{t}, \quad 0<p<1.$$
 In both cases $M_p(\mathscr{H}_\mu(f), r)\le M_p(f,r), \ r>0.$
	This shows that $\mathscr{H}_\mu(F^{p,q,\alpha})\subset F^{p,q,\alpha}$ with continuous inclusion for all $0<p,q\le \infty$ and $\alpha>0$.

	(vi) $\Longrightarrow$ (i) Just take $p=q$.
\end{proof}

\begin{corollary}\label{hauss1p} Let $\mu$ be a positive measure with $\mu_0=\int_0^\infty \frac{d\mu(t)}{t}=1.$ The following are equivalent:
\begin{itemize}
	\item[(i)] There exist $0< p < q\le \infty$ such that $\mathscr{H}_\mu:F^{p}_{\alpha}\to F^{q}_{\alpha}$ is bounded.
	\item[(ii)] $\mathscr{H}_\mu:F^{p}_{\alpha}\to F^{p}_{\alpha}$ is bounded for all $0< p\le \infty.$
\end{itemize}	
\end{corollary}
\begin{proof} (i) $\Longrightarrow$ (ii). Since $F^{\infty, p, \alpha} \subset F^p_\alpha \subset F^q_\alpha \subset F^{q, \infty, \alpha}$ with continuous inclusions then condition (iii) in Theorem \ref{mainteo} is satisfied.
	
	(ii) $\Longrightarrow$ (i). It follows trivially since $F^{p}_{\alpha}\subset F^{q}_{\alpha}$ for $p<q$.
\end{proof}

\section{Boundedness from large to small spaces}\label{sec:large-small}

We look for necessary or sufficient conditions for a Hausdorff operator to transform a Fock space into a smaller one. Sometimes we will write $\mathscr{H}_\mu(F^q_\alpha)\subset F^p_\alpha$ to mean that $\mathscr{H}_\mu:F^q_\alpha\to F^p_\alpha$ is a well defined and bounded operator.

We first observe that due to duality one actually has the following equivalence.
\begin{proposition} Let $1\le p <q \le \infty$. Then $\mathscr{H}_\mu:F^q_\alpha\to F^p_\alpha$ is bounded if and only if
	$\mathscr{H}_\mu:F^{p'}_\alpha\to F^{q'}_\alpha$ is bounded where $1/p+1/p'=1/q+1/q'=1.$
\end{proposition}
\begin{proof} Observe that if $\mathscr{H}_\mu:F^q_\alpha\to F^p_\alpha$ is bounded then $\mathscr{H}_\mu:F^q_\alpha\to F^q_\alpha$ is bounded and therefore, due to Theorem \ref{mainteo}, $\mathscr{H}_\mu(f)(z)=\sum_{n=0}^\infty \mu_n a_n z^n$ for $f\in F^q_\alpha$ with $f(z)=\sum_{n=0}^\infty a_n z^n$ where $\mu_n$ is a decreasing sequence of non-negative numbers.
	If $1< p<q<\infty$, due to the duality $(F^p_\alpha)^*=F^{p'}_\alpha$ one easily sees that   $\mathscr{H}_\mu:F^q_\alpha\to F^p_\alpha$ is bounded if and only if $\mathscr{H}_\mu^*=\mathscr{H}_\mu: F^{p'}_\alpha\to F^{q'}_\alpha$ is.

	In the cases $p=1$ and $q<\infty$ or  $p>1$ and $q=\infty$ we can use that $(F^1_\alpha)^*=F^{\infty}_\alpha$ and  $(f^\infty_\alpha)^*=F^{1}_\alpha$  and the fact that  if $\mathscr{H}_\mu:F^\infty_\alpha\to F^{t}_\alpha$ is continuous then so is $\mathscr{H}_\mu:F^{t'}_\alpha\to F^1_\alpha$ because it is the transposed map of the bounded operator $\mathscr{H}_\mu:f^\infty_\alpha\to F^{t}_\alpha$.
	
\end{proof}
We start with a result of self-improvement which follows from Kintchine's inequalities. Let $r_n(s)=sign(\sin(2^n \pi s)), s\in [0,1]$ be the Rademacher system. For each $f=\sum_{n=0}^\infty a_n u_n$ we denote $$\displaystyle R_sf=\sum_{n=0}^\infty a_n r_n(s) u_n.$$ Clearly $\|R_sf\|_{2,q,\alpha}= \|f\|_{2,q,\alpha}$ for any $0<q\le\infty$ and $0\le s\le 1.$

\begin{lem}\label{kint} Let $0<q<\infty$ and $f=\sum_{n=0}^\infty a_n u_n.$ Then
	$$\|f\|_{2,q, \alpha}\asymp \left(\int_0^1 \|R_sf\|^q_{q,\alpha}ds\right)^{1/q}.$$
\end{lem}
\begin{proof}   Due to Kintchine's inequalities we have
	$$
	\begin{array}{*2{>{\displaystyle}l}}
		\int_0^1\|R_sf\|^q_{q,\alpha}\ ds & = \alpha q\int_0^\infty \left(\int_0^1 M_q^q(R_sf,r)ds\right)e^{-\frac{\alpha q}{2}r^2} r dr \\ & \\
		& \asymp  \int_0^\infty M_2(f,r)^q e^{-\frac{\alpha q}{2} r^2} r dr\\ & \\
		& \asymp  \|f\|^q_{2,q,\alpha}.
	\end{array}$$
\end{proof}

\begin{theorem} \label{equiv} Let  $0<q\le \infty$.
	
	(i) Case $0<p<2$:
	$\mathscr{H}_\mu:F^{2,q,\alpha}\to F^{p}_\alpha$ is bounded if and only if $\mathscr{H}_\mu:F^{2,q,\alpha}\to F^{2,p,\alpha}$ is bounded.
	
	(ii) Case $2<p<\infty$:
	$\mathscr{H}_\mu: F^{p}_\alpha\to F^{2,q,\alpha}$ is bounded if and only if $\mathscr{H}_\mu:F^{2,p,\alpha}\to F^{2,q,\alpha}$ is bounded.

	In particular $$\mathscr{H}_\mu(F^{2}_{\alpha})\subset F^{1}_\alpha \Longleftrightarrow \mathscr{H}_\mu(F^{2}_{\alpha})\subset F^{2,1,\alpha} \hbox{ or }
	\mathscr{H}_\mu(F^{4}_{\alpha})\subset F^{2}_\alpha \Longleftrightarrow \mathscr{H}_\mu(F^{2,4,\alpha})\subset F^{2}_{\alpha}.$$
	
\end{theorem}
\begin{proof} (i) Since $p<2$ then $F^{2,p,\alpha}\subset F^p_\alpha$ and only the direct implication needs a proof. Assume that $\mathscr{H}_\mu(F^{2,q,\alpha})\subset F^{p}_\alpha$. Let $f\in F^{2,q,\alpha}$ be given and note that $R_sf\in F^{2,q,\alpha}$ with $\|R_s f\|_{2,q,\alpha}=\|f\|_{2,q,\alpha}$ for all $0\le s\le 1.$
	Hence from Lemma \ref{kint} and the fact $\mathscr{H}_\mu(R_s f)(z) = R_s(\mathscr{H}_\mu f)(z)$, we obtain
	\begin{eqnarray*}
		\|\mathscr{H}_\mu(f)\|^p_{2,p,\alpha}&\asymp& \int_0^1 \|R_s(\mathscr{H}_\mu f)\|_{p,\alpha}^p\ ds\\
		&=& \int_0^1 \|\mathscr{H}_\mu (R_sf)\|_{p,\alpha}^p\ ds\\
		&\le & C\int_0^1 \|R_sf\|_{2,q,\alpha}^p\ ds = C \|f\|^p_{2,q,\alpha}.
	\end{eqnarray*}
	
	(ii) Since $2<p<\infty$ then now $F^p_\alpha\subset F^{2,p,\alpha}$ again only the direct implication needs a proof. Assume that $\mathscr{H}_\mu(F^{p}_{\alpha})\subset F^{2,q,\alpha}$ and let $f\in F^{2,p,\alpha}$. Then  using Lemma \ref{kint} we have that $R_sf\in F^p_\alpha$\ a.e. $s\in (0,1).$

Hence $\|\mathscr{H}_\mu(R_s f)\|_{2,q,\alpha}\leq C\|R_s f\|_{p,\alpha}$\ a.e. $s\in (0,1)$ and we obtain
	\begin{eqnarray*}
		\|\mathscr{H}_\mu(f)\|^p_{2,q,\alpha}&=& \int_0^1 \|R_s(\mathscr{H}_\mu f)\|_{2,q,\alpha}^p\ ds\\
		&=& \int_0^1 \|\mathscr{H}_\mu (R_sf)\|_{2,q,\alpha}^p\ ds\\
		&\le & C\int_0^1 \|R_sf\|_{p,\alpha}^p\ ds\le C \|f\|^p_{2,p,\alpha}.
	\end{eqnarray*}
\end{proof}	

Let us discuss the inclusion $\mathscr{H}_\mu(F^q_\alpha)\subset F^p_\alpha$ for $p<q.$ We first remark some trivial necessary and sufficient conditions for such embedding to hold.

\begin{rmk}\label{rmk:5.1} Let $1\le p<q\le\infty$. Then
 $$\sum_{n=0}^\infty \mu_n (n+1)^{\frac{1}{2}(\frac{1}{p}-\frac{1}{q})}<\infty \Longrightarrow \mathscr{H}_\mu(F^{q}_{\alpha})\subset F^{p}_\alpha  \Longrightarrow \sup_n\mu_n (n+1)^{\frac{1}{2}(\frac{1}{p}-\frac{1}{q})}<\infty .$$
\end{rmk}
Indeed, note that if $f\in F^{q}_{\alpha}$ with $f=\sum_{n=0}^\infty a_n u_n$ then
$|a_n|r^{n}\le M_q(f,r)$ for all $n\in \mathbb N_0$ and $0<r<1$. Therefore
\begin{equation} \label{esti0}
|a_n|\|u_n\|_{q,\alpha}\le \|f\|_{q,\alpha}, \quad n\in \mathbb N_0.
\end{equation}
Now, using (\ref{esti0}) and (\ref{normaun})
\begin{eqnarray*}
\|\mathscr{H}_\mu(f)\|_{p,\alpha}&\le& \sum_{n=0}^\infty \mu_n |a_n| \| u_n\|_{p,\alpha}\\
&\le& C\sum_{n=0}^\infty \mu_n (n+1)^{\frac{1}{2}(\frac{1}{p}-\frac{1}{q})}|a_n| \| u_n\|_{q,\alpha}\\
&\le& C\Big(\sum_{n=0}^\infty \mu_n (n+1)^{\frac{1}{2}(\frac{1}{p}-\frac{1}{q})}\Big)\|f\|_{q,\alpha}.
\end{eqnarray*}
The other implication follows from (\ref{normaun}) since $\|\mathscr{H}_\mu(u_n)\|_{p,\alpha}\le C \|u_n\|_{q,\alpha}$.\qed

Another trivial estimates in the case $q=2$ are given in the following observation.
\begin{rmk}\label{rmk:5.2} Let $0< p\le 2$ and set $\gamma=\min\{p,1\}$. Then
 $$\displaystyle\sum_{n=0}^\infty\mu^{\frac{2\gamma}{2-\gamma}}_n (n+1)^{\frac{\gamma(2-p)}{2p(2-\gamma)}}<\infty \Longrightarrow \mathscr{H}_\mu(F^{2}_{\alpha})\subset F^p_\alpha \Longrightarrow  \sup_{n} \mu_n (n+1)^{\frac{(2-p)}{4p}}<\infty.$$
\end{rmk}
Indeed, for $\displaystyle f= \sum_{n=0}^\infty  a_n u_n$, and writing $\frac{1}{\gamma}=\frac{1}{2}+\frac{1}{u}$, or equivalently $u=\frac{2\gamma}{2-\gamma}$, since $\| u_n \|^\gamma_{p,\alpha}\asymp \| u_n \|_{2,\alpha} (n+1)^{\frac{1}{2p}-\frac{1}{4}}$ we have

\begin{eqnarray*}
	\|\mathscr{H}_\mu(f)\|^\gamma_{p,\alpha} & \le& \sum_{n=0}^\infty  \mu^\gamma_n |a_n|^\gamma \| u_n \|^\gamma_{p,\alpha} \\
&\asymp& \sum_{n=0}^\infty  \mu^\gamma_n |a_n|^\gamma \| u_n \|^\gamma_{2,\alpha}(n+1)^{(\frac{1}{2p}-\frac{1}{4})\gamma}\\
	& \le& \left(\sum_{n=0}^\infty \mu_n^u (n+1)^{u(\frac{1}{2p}-\frac{1}{4})}\right)^{\gamma/u} \left(\sum_{n=0}^\infty |a_n|^2 \| u_n \|^2_{2,\alpha}\right)^{\frac{\gamma}{2}}\\ & \\
	& \lesssim& \|f\|^\gamma_{2,\alpha}.
\end{eqnarray*}
The other implication follows as above evaluating on $u_n$. \qed

Let us analyze the extreme case $\mathscr{H}_\mu(F^{\infty}_{\alpha})\subset F^{1}_\alpha$ in the setting of Banach spaces. We shall get better necessary and sufficient conditions than the ones in the above remarks.
\begin{theorem} \label{case1} Let $\mu$ be a Borel positive measure defined on $(0,\infty)$. Then
$$\sum_{n=0}^\infty \mu_n<\infty \Longrightarrow \mathscr{H}_\mu(F^{\infty}_{\alpha})\subset F^{1}_\alpha  \Longrightarrow \sum_n\mu_n (n+1)^{-\frac{1}{2}}<\infty .$$
\end{theorem}
\begin{proof} Assume  that $\sum_{n=0}^\infty \mu_n<\infty$. This guarantees, due to Theorem \ref{mainteo}, that $\mathscr{H}_\mu(f)(z)=\int_1^\infty f(\frac{z}{t}) \frac{d\mu(t)}{t}$. Hence if $f\in F^\infty_\alpha$ and $\|f\|_{\infty,\alpha}\le 1$ then
 $$|\mathscr{H}_\mu(f)(z)|\le \int_1^\infty e^{\frac{\alpha|z|^2}{2 t^2}}\frac{d\mu(t)}{t}, \quad z\in \C.$$
 Therefore
 $$\|\mathscr{H}_\mu(f)\|_{1,\alpha}\le C \int_1^\infty \Big(\int_\C e^{-\frac{\alpha}{2}(1-\frac{1} {t^2})|z|^2}dA(z)\Big)\frac{d\mu(t)}{t}\le C\int_1^\infty (1-\frac{1} {t^2})^{-1}\frac{d\mu(t)}{t}.$$

 Since $\int_1^\infty (1-\frac{1} {t^2})^{-1}\frac{d\mu(t)}{t}=\sum_{n=0}^\infty \mu_{2n}<\infty$ we have the first implication.

 Assume now that  $\mathscr{H}_\mu(F^{\infty}_{\alpha})\subset F^{1}_\alpha$, which, due to Theorem \ref{mainteo}, allows us to say that $\mathscr{H}_\mu(f)=\sum_{n=0}^\infty \mu_n a_n u_n\in F^1_\alpha$ for any $f=\sum_{n=0}^\infty  a_n u_n\in F^\infty_\alpha.$

  Let us mention now the following fact on Taylor coefficients of functions in $F^1_\alpha$, (see \cite{BG} for $\alpha=2$ and \cite[Theorem 2]{T} for general values of $\alpha$),
\begin{equation} \label{esti1}
\sum_{n=0}^\infty |a_n|\|u_n\|_{1,\alpha}(n+1)^{-1/2}\le C \|f\|_{1,\alpha}.
\end{equation}
Using duality one also has that
\begin{equation} \label{esti2}
 \|f\|_{\infty,\alpha}\le C\sup_{n\ge 0} |a_n|\|u_n\|_{\infty,\alpha}(n+1)^{1/2}.
\end{equation}
 Therefore, using (\ref{esti1}) and (\ref{esti2}) and the fact $\|u_n\|_{1,\alpha}(n+1)^{-1/2}\asymp \|u_n\|_{\infty,\alpha}$ we obtain for any sequence $(a_n)$ that
\begin{eqnarray*}
\sum_{n=0}^\infty \mu_n|a_n|\|u_n\|_{\infty,\alpha}&\le& C \|\mathscr{H}_\mu(f)\|_{1,\alpha}\\
&\le &C \|f\|_{\infty,\alpha}\\
&\le& C(\sup_{n\ge 0} |a_n|\|u_n\|_{\infty,\alpha}(n+1)^{1/2}).
\end{eqnarray*}
This implies that $\sum_{n=0}^\infty \mu_n(n+1)^{-1/2}<\infty$ and the proof is complete.

\end{proof}

\par\medskip
Remarks \ref{rmk:5.1} and \ref{rmk:5.2} can be actually improved for all values of $p <q$ using the norm of the dilation operator and results on Taylor coefficients. For $t > 1$ the dilation operator $D_{1/t}$ is defined by $\left(D_{1/t}f\right)(z) = f\left(\frac{z}{t}\right).$
\par\medskip
Using that $\mu_0=1$ and $\mathcal H_\mu(u_0)=u_0$ we obtain that $\|D_{1/t}\|_{F^{p,q,\alpha}\to F^{p,q,\alpha}}\ge 1$. On the other hand, since $M_p(D_{1/t}f,r) = M_p(f, \frac{r}{t}) \leq M_p(f,r)$ for all $t > 1$ we conclude
	$$
	\|D_{1/t}\|_{F^{p,q,\alpha}\to F^{p,q,\alpha}} =1\ \ \forall t > 1.$$ We now present some estimates for the norm of dilation operators when acting from a Fock space into a smaller one.

\begin{lem} \label{dilation} Let $0<p< q \le \infty$  and $t>1$. Then $D_{1/t}: F^q _\alpha \to F^p_\alpha$ is bounded and  $$C_1\left(1-\frac{1}{t^2}\right)^{\frac{1}{2q}-\frac{1}{2p}}\le \|D_{1/t}\|_{F^q _\alpha\to F^p_\alpha}\le C_2t^{2/q }\left(1-\frac{1}{t^2}\right)^{\frac{1}{q}-\frac{1}{p}}$$
	for some constants $C_1,C_2$ independent of $t$.
\end{lem}
\begin{proof}
	Let $f\in F^q _\alpha$ and assume $q <\infty$.   Then, denoting $1/p-1/q =1/u$, H\"older's inequality gives
	\begin{eqnarray*}
		\|D_{1/t}f\|_{p,\alpha}&=& \left(\frac{\alpha p}{2\pi}\int_\C |f(z/t)|^p e^{-\frac{\alpha}{2} p|z|^2}dA(z)\right)^{1/p}\\
		&=& \left(\frac{\alpha p}{2\pi}t^2\int_\C |f(z)|^p e^{-\frac{\alpha}{2} p |z|^2}e^{-\frac{\alpha}{2} p (t^2-1)|z|^2}dA(z)\right)^{1/p}\\
		&\lesssim & t^{2/p}\left(\int_\C |f(z)|^q  e^{-\frac{\alpha}{2} q  |z|^2} dA(z)\right)^{1/q }  \left(\int_\C e^{-\frac{\alpha}{2} u (t^2-1)|z|^2}dA(z)\right)^{1/u}\\
		&\lesssim & \|f\|_{q ,\alpha}t^{2/p}  \left(\frac{1}{t^2-1}\right)^{1/u}\\
		&\asymp & \|f\|_{q ,\alpha}t^{2/q }  \left(1-\frac{1}{t^2}\right)^{1/q -1/p}.
	\end{eqnarray*}
Hence we obtain the upper estimate.

To handle the lower one, just observe that
	 $D_{1/t}u_0 = u_0$ which gives $\|D_{1/t}\|_{F^q _\alpha\to F^p_\alpha}\geq 1$ for all $t > 1.$ Therefore we need to deal with  $1<t< 3/2$. Using now that $D_{1/t}u_n=\frac{1}{t^n}u_n$ for $n\in \mathbb N$ we get $$\frac{\|D_{1/t}u_n\|_{p,\alpha}}{\|u_n\|_{q ,\alpha}}\asymp \frac{1}{t^n} (n+1)^{\frac{1}{2p}-\frac{1}{2q }}.$$ We put $a = \frac{1}{2p}-\frac{1}{2q} > 0.$ For each $1 < t < \frac{3}{2}$ we consider $n:=\left[\frac{1}{t-1}\right].$ Then
	$$
	\|D_{1/t}\|_{F^q _\alpha\to F^p_\alpha} \gtrsim \frac{n^a}{t^{n}} \gtrsim   \frac{\left(\frac{1}{t-1}\right)^a}{t^{\frac{1}{t-1}}} \asymp \left(\frac{1}{t-1}\right)^a\asymp \left(1 - \frac{1}{t^2}\right)^{-a}.$$

	The case   $q =\infty$  follows similarly.
\end{proof}

To work with the inclusion $\mathscr{H}_\mu(F^q_\alpha)\subset F^p_\alpha$ in the case $1\le p\le 2<q\le \infty$ we can use results on Taylor coefficients of functions $f = \sum_{n=0}^{\infty} a_n u_n$ in Fock spaces (\cite{BG, T}).
Recall (see \cite[Theorem 4]{T}) that for $0< p\le 2$ there exists $C>0$ such that
\begin{equation} \label{eqpp}\|f\|_{p,\alpha}\le C\|(a_n \Big(\frac{n!}{\alpha^n}\Big)^{1/2} (n+1)^{\frac{1}{2}(\frac{1}{p} - \frac{1}{2})})\|_{\ell_p}
\end{equation}
and  its dual version, for $2\le q<\infty$
\begin{equation} \label{eqqq}  \|(a_n \Big(\frac{n!}{\alpha^n}\Big)^{1/2} (n+1)^{\frac{1}{2}(\frac{1}{q} - \frac{1}{2})})\|_{\ell_q}\le C\|f\|_{q,\alpha}.
\end{equation}
We shall mention some estimates for the converse inequalities. They are some improvements respect to the ones given in \cite{T}.
\begin{lem} \label{coefi} Let $1\le p\le 2\le q\le\infty$ and $f(z)=\sum_{n=0}^\infty a_n z^n$. Then there exists $C>0$ such that \begin{equation}\label{equa1}\|(a_n \sqrt{\frac{n!}{\alpha^n}}(n+1)^{\frac{1}{2}(\frac{1}{2}-\frac{1}{p})})\|_{\ell_p}\le C \|f\|_{p,\alpha},\end{equation}
	\begin{equation}\label{equa2}\|f\|_{q,\alpha}\le C\|(a_n \sqrt{\frac{n!}{\alpha^n}}(n+1)^{\frac{1}{2}(\frac{1}{2}-\frac{1}{q})})\|_{\ell_q}.\end{equation}
\end{lem}
\begin{proof} The cases $p=1$ and $q=\infty$ correspond to (\ref{esti1}) and (\ref{esti2}) respectively. Let $T$ be the linear operator defined in $H(\C)$ given by  $$T(f)= ( a_n \sqrt{\frac{n!}{\alpha^n}} )_{n\ge 0}.$$ To show (\ref{equa1}) we consider the weight $W= ((n+1)^{-1/4})_{n\ge 0}$ and observe that (\ref{esti1}) gives $T$ is bounded from $F^1_\alpha$ into $\ell_1(W)$. Since clearly $T$ is also bounded from $F^2_\alpha$ into $\ell_2$, then using interpolation with change of measures due to Stein-Weiss (see \cite{SW}), we obtain that $T$ is bounded from $[F^1_\alpha,F^2_\alpha]_\theta= F^p_\alpha$ for $1/p=1-\theta/2$ into $[\ell_1(W),\ell_2]_\theta=\ell_p(U)$  where $U= ((n+1)^{-p(1-\theta)/4})_{n\ge 0}$.
	In particular, since $\frac{1}{p}=1-\frac{\theta}{2}$ we have that $p(1-\theta)=2-p$ and we obtain that
	$$\sum_{n=0}^\infty |a_n|^p \left(\frac{n!}{\alpha^n}\right)^{\frac{p}{2}}(n+1)^{-\frac{1}{4}(2-p)}\le C\|f\|^p_{p,\alpha}.$$ This completes the proof of (\ref{equa1}). Now (\ref{equa2}) follows by duality.
\end{proof}

Now we get the extension of Theorem \ref{case1} to other values of $p<q$.
\begin{theorem} \label{teopq} Let $1\le p\le 2 \le q\leq \infty$ with $p<q$ and let $\mu$ be a positive Borel measure on $(0,\infty)$.
	
	(i) If $\displaystyle\int_1^\infty t^{\frac{2}{q}-1}d\mu(t)<\infty$ and $\displaystyle\sum_{n=0}^\infty \mu_{n} (n+1)^{\frac{q-p}{pq}-1}<\infty $ then $
	\mathscr{H}_\mu:F^q_\alpha\to F^p_\alpha$ is bounded.
	
	(ii) If $\mathscr{H}_\mu:F^q_\alpha\to F^p_\alpha$ is bounded then $\displaystyle\sum_{n=0}^\infty \mu_n^{\frac{pq}{q-p}}(n+1)^{-1/2}<\infty$.
\end{theorem}
\begin{proof} We only discuss the case $1\le p\le 2 \le q < \infty.$ The case $q=\infty$ is similar and left to the reader.
	\par\medskip
	(i) Since the series $\sum_{n=0}^\infty (n+1)^{\frac{q-p}{pq}-1}\mu_{n}$ is convergent then, arguing as in Theorem \ref{mainteo}, $\mu(0,1)=0$ and
	$\mathscr{H}_\mu(f)(z)=\int_1^\infty D_{1/t}f(z) \frac{d\mu(t)}{t}.$ Therefore, writing $\gamma = \frac{1}{p}-\frac{1}{q}>0$ and using
	Lemma \ref{dilation}, the assumption $q\ge 2,$ and
	$$
	\displaystyle \frac{1}{(1-s)^{\gamma}}\asymp \sum_{n=0}^\infty (n+1)^{\gamma-1}s^n\ \mbox{for}\  0<s<1,$$ we get
	\begin{eqnarray*}\|\mathscr{H}_\mu(f)\|_{p,\alpha}&\le& \int_1^\infty \|D_{1/t}f\|_{p,\alpha} \frac{d\mu(t)}{t}\\
		&\le& C \|f\|_{q,\alpha}\int_1^\infty t^{2/q}(1-\frac{1}{t^2})^{-\gamma} \frac{d\mu(t)}{t}\\
		&\le& C \|f\|_{q,\alpha} \int_1^\infty \sum_{n=0}^\infty (n+1)^{\gamma-1}\frac{t^{2/q}}{t^{2n}}\frac{d\mu(t)}{t}\\
		&\le& C \|f\|_{q,\alpha}\left(\int_1^\infty t^{\frac{2}{q}-1}d\mu(t)+ \int_1^\infty \sum_{n=1}^\infty (n+1)^{\gamma-1}\frac{t}{t^{n}}\frac{d\mu(t)}{t}\right)\\
		&\le& C \|f\|_{q,\alpha} \left(\int_1^\infty t^{\frac{2}{q}-1}d\mu(t)+ \sum_{n=1}^\infty (n+1)^{\gamma-1}\mu_{n-1}\right).
	\end{eqnarray*}
	
	(ii) Assume now that $\mathscr{H}_\mu: F^q_\alpha\to F^p_\alpha$ is bounded.
	Combining (\ref{equa1}) and (\ref{equa2}) in Lemma \ref{coefi} we obtain that
	\begin{eqnarray*}
		&&\|(\mu_n a_n \Big(\frac{n!}{\alpha^n}\Big)^{1/2} (n+1)^{\frac{1}{2}(\frac{1}{2}-\frac{1}{p})})\|_{\ell_p}\\
		&\le& C\|\mathscr{H}_\mu(f)\|_{p,\alpha}\le C\|f\|_{q,\alpha}\\
		&\le& C\|( a_n \Big(\frac{n!}{\alpha^n}\Big)^{1/2} (n+1)^{\frac{1}{2}(\frac{1}{2}-\frac{1}{q})})\|_{\ell_q} .
	\end{eqnarray*}
	Therefore $(\mu_n (n+1)^{\frac{1}{2}(\frac{1}{q}-\frac{1}{p})})_{n\geq 0}\in \ell_\gamma$ for $\frac{1}{\gamma}=\frac{1}{p}-\frac{1}{q}$.
	
	We then conclude that $\displaystyle\sum_{n=0}^\infty \mu_n^{\frac{pq}{q-p}}(n+1)^{-1/2}<\infty$.
\end{proof}

In relation to Theorem \ref{teopq} (i) it is worth mentioning that the integral condition is easier to hold for large values of $q$ but in such a case the condition that refers to the convergence of the series is more demanding.
\par\medskip
The case  $q=\infty$ in Theorem \ref{teopq} gives the following result.
\begin{corollary}\label{th:from-inf-to-1} Let $\mu$ be a positive Borel measure and $1\le p\le 2$.
	$$\sum_{n=0}^\infty \mu_n (n+1)^{-1/p'}<\infty \Longrightarrow \mathscr{H}_\mu(F^\infty_\alpha)\subset F^p_\alpha \Longrightarrow \sum_{n=0}^\infty \mu^p_n(n+1)^{-1/2}<\infty.$$
\end{corollary}

The above results can be easily improved in the case $q=2$.

\begin{theorem} Let $\mu$ a Borel positive measure on $(0,\infty)$ and $0< p<2.$
	\begin{itemize}
		\item[(i)] If $\displaystyle\sum_{n=0}^\infty\mu^{\frac{2p}{2-p}}_n\sqrt{n+1} <\infty$ then $\mathscr{H}_\mu:F^2_\alpha\to F^p_\alpha$ is bounded.
		\item[(ii)] If $\mathscr{H}_\mu:F^2_\alpha\to F^p_\alpha$ is bounded then $\displaystyle\sum_{n=0}^\infty\mu^{\frac{2p}{2-p}}_n <\infty$.
	\end{itemize}	
\end{theorem}
\begin{proof}  (i) Let $\gamma=\frac{2p}{2-p}$ be such that $\frac{1}{p} = \frac{1}{2}+ \frac{1}{\gamma}$ and $f(z)=\sum_{n=0}^\infty a_n z^n\in F^2_\alpha$. Using (\ref{eqpp}) we have
	\begin{eqnarray*}
		\|\mathscr{H}_\mu(f)\|_{p,\alpha}&\le& C \|(\mu_n a_n \sqrt{\frac{n!}{\alpha^n}} (n+1)^{\frac{1}{2p}-\frac{1}{4}})\|_{\ell_p}\\
		&\le& C \|(\mu_n  (n+1)^{\frac{1}{2p}-\frac{1}{4}})\|_{\ell_\gamma}\|(a_n \sqrt{\frac{n!}{\alpha^n}} )\|_{\ell_2}\\
		&\le& C \left(\sum_{n=0}^\infty\mu_n^\gamma  \sqrt{n+1}\right)^{1/\gamma}\|f \|_{2,\alpha}.
	\end{eqnarray*}
	
	(ii)  Assume now that $\mathscr{H}_\mu: F^2_\alpha\to F^p_\alpha$ is bounded. Using Lemma \ref{equiv} then $\mathscr{H}_\mu: F^2_\alpha\to F^{2,p,\alpha}$ is bounded.
	
	Hence, using that $\sum_{n=0}^\infty \frac{r^n e^{-r}}{n!}=1$ and $p<2$ we have
	\begin{eqnarray*}
		\|\mathscr{H}_\mu(f)\|^p_{2,p,\alpha}&\asymp& \int_0^\infty \left(\sum_{n=0}^\infty \mu_n^2 |a_n|^2 r^{2n} e^{-\alpha r^2}\right)^{p/2}rdr\\
		&\asymp& \int_0^\infty \left(\sum_{n=0}^\infty \mu_n^2 \frac{|a_n|^2}{\alpha^n} n! \frac{r^n e^{-r}}{n!}\right)^{p/2}dr\\
		&\ge&  C\sum_{n=0}^\infty \mu_n^p |a_n|^p\left(\frac{n!}{\alpha^{n}}\right)^{p/2}.
	\end{eqnarray*}
	Now using that  $\|\mathscr{H}_\mu(f)\|_{2,p,\alpha}\le C \left(\sum_{n=0}^\infty |a_n|^2\frac{n!}{\alpha^{n}}\right)^{1/2}$ for all sequences $(a_n)$ with
	$\sum_{n=0}^\infty |a_n|^2\frac{n!}{\alpha^{n}} < \infty$ we obtain $(\mu_n)\in \ell_\gamma$ for $\gamma$ such that $\frac{1}{p}=\frac{1}{2}+\frac{1}{\gamma}$ and the proof of this implication is complete.
\end{proof}

\begin{corollary} \label{case2} Let $\mu$ be a positive Borel measure. Then
	$$\displaystyle\sum_{n=0}^\infty\mu^2_n \sqrt{n+1}<\infty \Longrightarrow \mathscr{H}_\mu(F^2_\alpha)\subset F^{2,1,\alpha }$$
	$$\Longleftrightarrow \mathscr{H}_\mu(F^2_\alpha)\subset F^1_\alpha\Longleftrightarrow \mathscr{H}_\mu(F^\infty_\alpha)\subset F^{2}_{\alpha}
	\Longrightarrow \displaystyle\sum_{n=0}^\infty\mu_n^2<\infty.$$
\end{corollary}

We would like now to analyze the inclusions between Fock and mixed norm Fock spaces. Recall that for $q<p$ we have $F^{p,q,\alpha}\subset F^q_\alpha$
while for $q>p$ then $F^q_\alpha \subset F^{p,q,\alpha}$. We shall make use of the following general lemma.

\begin{lem} \label{explema} Let $\delta\in \mathbb R$ and let $(\gamma_n)_n$ be a sequence of non-negative real numbers such that $(\gamma_n(n+1)^\beta)_n$ is monotonic for some $\beta\in \mathbb R$. Then the following are equivalent:

	(i) $\sup_n \gamma_n (n+1)^{\delta}<\infty.$
	
	(ii) $\sum_{n=0}^\infty \frac{\gamma_n (n+1)^{\delta}x^{n}}{n!}\le Ce^{x} \ \ \forall x>0$ for some constant $C > 0.$
	
\end{lem}
\begin{proof}
	(i) $\Longrightarrow$ (ii) is obvious.
	(ii) $\Longrightarrow$ (i)  Let $a>1$. Then multiplying by $e^{-ax}$ and integrating over $(0,\infty)$ one has
	$$\sum_{n=0}^\infty \frac{\gamma_n (n+1)^{\delta}}{n!}\int_0^\infty x^{n}e^{-ax}dx \le C\int_0^\infty e^{-(a-1)x}dx= \frac{C}{a-1}.$$
	Therefore, for $r=1/a$,
	$$\sum_{n=0}^\infty \gamma_n (n+1)^{\delta}r^{n+1}\le C \frac{r}{1-r}, \quad 0<r<1.$$
	Hence
	$$(1- \frac{1}{n})^{2n+1}\sum_{k=n}^{2n} \gamma_k (k+1)^{\delta}\le  Cn, \quad n\in \mathbb N.$$
	
	Assuming that $(\gamma_n (n+1)^{\beta})_n$ is increasing one obtains
	$$\gamma_{n}(n+1)^{\delta +1}\lesssim \gamma_{n }(n+1)^{\beta}\sum_{k=n}^{2n}  (k+1)^{\delta-\beta}\le  \sum_{k=n}^{2n}\gamma_k(k+1)^\delta\lesssim n+1.$$
	
	In case that $(\gamma_n (n+1)^{\beta})_n$ is decreasing we use a similar argument to obtain the result.
\end{proof}

\begin{theorem}\label{th:fromF1toFinf1} Let $\mu$ be a positive measure.
	
	$\mathscr{H}_\mu:F^{1}_\alpha\to F^{\infty,1,\alpha}$ is bounded if and only if $\sup_n \mu_{n}\sqrt{n+1}<\infty$.
	
\end{theorem}
\begin{proof}
	Assume that $\mathscr{H}_\mu$ is bounded from $F^1_\alpha$ into $F^{\infty,1,\alpha}$. In particular, there exists $C>0$ such that
	$$
	\|\mathscr{H}_\mu(K_{\alpha}(\cdot, a)\|_{\infty,1,\alpha}\le C e^{\frac{\alpha}{2}a^2},\quad a>0, $$ which gives
	$$
	\alpha\int_0^\infty\left(\sum_{n=0}^\infty \frac{\mu_n \alpha^na^n}{n!}r^ne^{-\frac{\alpha}{2} r^2}\right)r\ dr\le Ce^{\frac{\alpha}{2}a^2}, \quad a>0.$$
	Hence, replacing $\frac{\alpha}{2}a^2 = x$, we obtain
	$$ \sum_{n=0}^\infty \frac{\mu_n x^{n/2}2^n}{n!}\Gamma(\frac{n}{2}+1)\le Ce^x, \quad x>0.$$ In particular
	$$ \sum_{n=0}^\infty \mu_{n}\sqrt{n+1}\frac{x^{n/2}}{\Gamma(\frac{n}{2}+1)}\le Ce^{x}, \quad x>0.$$
	Clearly $\displaystyle \sum_{n=0}^\infty \mu_{2n}\sqrt{n+1}\frac{x^{n}}{n!}\le Ce^{x}$ for all $x>0$
	and  then applying  Lemma \ref{explema} we obtain $\sup_{n\in \mathbb N} \mu_{2n}\sqrt{n+1}<\infty.$ Since the sequence $(\mu_n)$ is decreasing we also obtain $\sup_{n\in \mathbb N} \mu_{2n+1}\sqrt{n+1}<\infty$ and we get the direct implication.
	
	Conversely, we assume that $\mu_n\sqrt{n+1}\le C$ for all $n\ge 0$.
	Using (\ref{esti1})  we have
	\begin{eqnarray*}
		\|\mathscr{H}_\mu(f)\|_{\infty,1,\alpha}&\le & \alpha\int_0^\infty(\sum_{n=0}^\infty \mu_n |a_n|r^n e^{-\frac{\alpha}{2}r^2})rdr\\
		&\le &\sum_{n=0}^\infty \mu_n (\frac{2}{\alpha})^{n/2}|a_n|\Gamma(\frac{n}{2}+1)\\
		&\le & C\sum_{n=0}^\infty \mu_n |a_n|\sqrt{\frac{n!}{\alpha^n}}(n+1)^{1/4}\\
		&\le & C\sum_{n=0}^\infty  |a_n|\sqrt{\frac{n!}{\alpha^n}}(n+1)^{-1/4}\\
&\asymp & C\sum_{n=0}^\infty  |a_n|\|u_n\|_{1,\alpha}(n+1)^{-1/2}\\
		&\le & C \|f\|_{1,\alpha}.
	\end{eqnarray*}
\end{proof}

Let us now get some result for the inclusion $\mathscr{H}_\mu(F^q_\alpha)\subset F^{2,q,\alpha}$ in the case $0<q<2$.
\begin{theorem} \label{main2}Let $\int_0^\infty \frac{d\mu(t)}{t}=1$ and let $0<q< 2$.
	
	(i) If  $\sup_n \mu_n (n+1)^{\frac{2-q}{2q}}<\infty$ then $\mathscr{H}_\mu(F^q_\alpha)\subset F^{2,q,\alpha}$.
	
	(ii) If $\mathscr{H}_\mu(F^q_\alpha)\subset F^{2,q,\alpha}$ then
	$\sup_n \mu_n (n+1)^{\frac{2-q}{4q}}<\infty$.
\end{theorem}
\begin{proof}
	(i) Assume that $\sup_n \mu_n (n+1)^{\frac{2-q}{2q}}<\infty.$ Then $\sup_{n}\mu_n<\infty$ and, by Theorem \ref{mainteo}, $\mathscr{H}_\mu$ is bounded from $F^q_\alpha$ into itself. Also $\mathscr{H}_\mu(f)=\sum_{n=0}^\infty \mu_n a_nu_n$ for $f = \sum_{n=0}^\infty a_n u_n\in F^q_\alpha.$
We first deal with the case $0<q<1$.
Notice that
	$$\sum_{k=0}^n \mu_k^2 (k+1)^{2/q}\le C \sum_{k=0}^n (k+1) \le C n^2, \quad n\in \mathbb N$$
	and hence $(\mu_n)$ defines a multiplier from $H^q$ into $H^2$ (see  \cite[Theorem 6.6]{D}) which, due to the fact $(\mathscr{H}_\mu(f))_r=\mathscr{H}_\mu(f_r),$ where $f_r(z) = f(rz),$ implies
	$$M_2(\mathscr{H}_\mu(f), r)\le C M_q(f,r), \quad 0<r<\infty.$$
	Therefore $\mathscr{H}_\mu(F^q_\alpha)\subset F^{2,q,\alpha}$.

The case $q=1$ is a consequence of Theorem \ref{th:fromF1toFinf1} since $\mathscr{H}_\mu(F^1_\alpha)\subset F^{\infty,1,\alpha}\subset F^{2,1,\alpha}$.

Assume now that $1<q<2$ and let $f\in F^q_\alpha$.
\begin{eqnarray*}
\|\mathscr{H}_\mu(f)\|^q_{2,q,\alpha}&\asymp& \int_0^\infty (\sum_{n=0}^\infty \mu_n^2 |a_n|^2 r^{2n} e^{-\alpha r^2})^{q/2} r dr\\
&\asymp& \int_0^\infty (\sum_{n=0}^\infty \mu_n^2 \frac{|a_n|^2}{\alpha^n} s^{n} e^{-s})^{q/2}ds\\
&\le& C\int_0^\infty (\sum_{n=0}^\infty \mu_n^q \frac{|a_n|^q}{\alpha^{nq/2}} s^{nq/2} e^{-sq/2})ds\\
&\le& C\sum_{n=0}^\infty \mu_n^q \frac{|a_n|^q}{\alpha^{nq/2}}(\frac{2}{q})^{nq/2} \Gamma(\frac{nq}{2}+1)\\
&\le& C\sum_{n=0}^\infty \mu_n^q |a_n|^q \Big(\sqrt{\frac{n!}{\alpha^{n}}}\Big)^q(n+1)^{\frac{1}{2}-\frac{q}{4}}.
\end{eqnarray*}

Since $\mu_n^q\le C (n+1)^{q/2-1}$ then using (\ref{equa1}) we get
$$\|\mathscr{H}_\mu(f)\|^q_{2,q,\alpha}\le C \sum_{n=0}^\infty  |a_n|^q \Big(\sqrt{\frac{n!}{\alpha^{n}}}\Big)^q(n+1)^{\frac{q}{4}-\frac{1}{2}}\le C\|f\|^q_{q,\alpha}.$$
	
	(ii) Assume that $\mathscr{H}_\mu(K_\alpha(\cdot,a))\in F^{2,q,\alpha}$ with $$\|\mathscr{H}_\mu(K_\alpha(\cdot,a))\|_{2,q,\alpha}\leq C \|K_\alpha(\cdot,a)\|_{q,\alpha}.$$
	Hence 
	$$\int_0^\infty \left(\sum_{n=0}^\infty \frac{\mu^2_n \alpha^{n} a^{2n}}{n!} \frac{\alpha^n r^{2n}e^{-\alpha r^2}}{n!}\right)^{q/2} r dr\le Ce^{\frac{\alpha q}{2}a^2},\ a > 0.$$

	Using that $0<q\le 2$, $\sum_{n=0}^\infty \frac{\alpha^n r^{2n}e^{-\alpha r^2}}{n!} =1 $ and $\int_0^\infty\frac{\alpha^n r^{2n}e^{-\alpha r^2}}{n!}rdr =\frac{1}{2\alpha}$ we have
	$$ \sum_{n=0}^\infty \left(\frac{\mu^2_{n} (\alpha a^2)^n}{n!}\right)^{q/2} \le Ce^{\frac{\alpha q}{2}a^2}, \quad a>0.$$ Finally
	
	$$ \sum_{n=0}\mu_n^q (n+1)^{\frac{1}{2}-\frac{q}{4}}\frac{x^{nq/2}}{\Gamma(\frac{nq}{2}+1)}\asymp\sum_{n=0}^\infty (\frac{\mu^2_{n} (\frac{2}{q})^n }{n!})^{q/2} x^{nq/2} \le Ce^{x}, \quad x>0.$$

	After multiplying by $e^{-x/r}, r<1,$ and integrating over $(0,\infty)$ we get
	
	$$\sum_{n=0}^\infty \mu^q_{n}(n+1)^{\frac{1}{2}-\frac{q}{4}}r^{\frac{nq}{2}+1} \le C\frac{r}{1-r}, \quad 0<r<1.$$ Arguing as in the proof of Lemma \ref{explema} we get the conclusion.
\end{proof}

\section{Compactness of Hausdorff operators}

 We recall that a linear operator $T:E\to F$ between Fréchet spaces is said to be compact if there exists a neighbourhood $U$ of the origin in $E$ such that $T(U)$ is relatively compact in $F.$

\begin{proposition} Let $\mu \ne 0$ be a measure such that $\sup_n \sqrt[n]{\mu_n}<\infty.$ Then the Hausdorff operator $\mathscr{H}_\mu:H({\mathbb C})\to H({\mathbb C})$ is not compact.
\end{proposition}	
\begin{proof}
	Let us assume  $\mathscr{H}_\mu(U)$ is relatively compact for some neighbourhood $U$ of the origin in $H({\mathbb C})$. We will show that
	\begin{equation}\label{eq:cond-necesaria}
		\displaystyle\lim_n \sqrt[n]{\mu_n} = 0.\end{equation}
	There exist $ r > 0$ and $\varepsilon > 0$ such that $$\left\{f\in H({\mathbb C}): \ M_\infty(f,r) \leq \varepsilon \right\}\subset U.$$ Since $\mathscr{H}_\mu(U)$ is a bounded set then
	$$
	C_R:=\sup_{f\in U} M_\infty(\mathscr{H}_\mu f, R) < \infty\ \ \forall R > 0.$$ For every $f\in H({\mathbb C}), f\neq 0,$ we have $\frac{\varepsilon}{M_\infty(f,r)} f\in U,$ hence
		$$
	M_\infty(\mathscr{H}_\mu f, R)\leq C_R \varepsilon^{-1}M_\infty(f,r)\ \ \forall f\in H({\mathbb C})\ \ \forall R > 0.$$ In particular, taking $f = u_n,$ we conclude $\mu_n R^n\leq C_R \varepsilon^{-1}r^n$ for every $n\in {\mathbb N}$ and for every $R > 0,$ from where (\ref{eq:cond-necesaria}) follows. Now for every $0 < a < b$ we observe that
	$$
	\mu_n \geq \int_a^b\frac{d\mu(t)}{t^{n+1}} \geq \frac{\mu\left([a,b]\right)}{b^{n+1}}.$$ From (\ref{eq:cond-necesaria}) we get $\mu\left([a,b]\right) = 0$, leading to a contradiction.
\end{proof}	

\begin{lem}\label{compactinclusion2}Let $\beta < \alpha$ and $0<p,q\le \infty$. Then $F^{p,q,\beta}\subset F^{p,q,\alpha}$ is a compact inclusion.
\end{lem}
\begin{proof} We will only discuss the case $q<\infty,$ since the case $q=\infty$ is quite similar.
	Let $B$ denote the closed unit ball in $F^{p,q,\beta}.$  Let us see that $B$ is a bounded set in the Fréchet space $H({\mathbb C}).$ Indeed, from \cite[Theorem 5.9]{D} there exists $C> 0$ such that
	$$
	M_\infty(f,r)\leq CM_p(f,2r) , \quad r>0$$ and then for a given $r>0$ we have, using the inclusion $F^{p,q,\beta}\subset F^{p,\infty,\beta}$,
	$$\sup_{f\in B} M_\infty(f,r) \le C \sup_{f\in B} M_p(f,2r)\le C \sup_{f\in B} \|f\|_{p,\infty,\beta}e^{2\beta r^2} \le Ce^{2\beta r^2}.$$
	
	For every $\varepsilon > 0$ we choose $R > 0$ such that
	$$
	e^{-q\frac{R^2}{2}(\alpha - \beta)} \leq \frac{\varepsilon}{2^{q+1}}.$$ (In the case $q = \infty$ we choose $R$ so that $e^{-\frac{R^2}{2}(\alpha - \beta)} \leq \frac{\varepsilon}{4}$). Then for every $f\in 2B$ we have
	$$
	\|f\|^q_{p,q,\alpha} = \alpha q\int_0^\infty M^q_p(f,r) e^{-q\frac{r^2}{2}\beta}e^{-q\frac{r^2}{2}(\alpha - \beta)}rdr.$$ Hence
	$$
	\|f\|^q_{p,q,\alpha} \leq \alpha q\int_0^R M^q_p(f,r)e^{-q\frac{\alpha}{2}r^2} rdr + \frac{\varepsilon}{2}\le CM_p^q(f,R)+\frac{\varepsilon}{2}$$ for every $f\in 2B.$  Using now that $B$ is bounded in $H({\mathbb C})$, by Montel's theorem we can find a finite subset ${\mathcal F}\subset B$ such that for every $f\in B$ there exists $g\in {\mathcal F}$ satisfying $M^q_\infty(f-g, R) \leq \frac{\varepsilon}{2C}.$ Since $f-g\in 2B$ we can use the previous estimates to conclude
	$$
	\|f - g\|^q_{p,q,\alpha}\leq CM_\infty^q(f-g,R) +  \frac{\varepsilon}{2} \leq \varepsilon.$$ Consequently $B$ is relatively compact in $F^{p,q,\alpha}.$
\end{proof}

\begin{rmk}{\rm Let $0 < \beta < \alpha$ and $\lambda > 1$ so that $\lambda^2 \beta = \alpha.$ We consider the measure $\mu = \lambda \delta_\lambda.$ Then
		$$
		\left(\mathscr{H}_\mu f\right)(z) = \int_0^\infty f\left(\frac{z}{t}\right)\frac{d\mu(t)}{t} = f\left(\frac{z}{\lambda}\right).$$ Since
		$$
		\|f\|_{p,\infty,\alpha} = \sup_{r > 0}M_p(f,\frac{r}{\lambda})e^{-\frac{\beta}{2}r^2}$$ it turns out that $\mathscr{H}_\mu:F^{p,\infty,\alpha}\to F^{p,\infty,\beta}$ is an isomorphism. Consequently, the compactness of the inclusion map $J: F^{p,\infty,\beta}\to F^{p,\infty,\alpha}$ is equivalent to the compactness of $\mathscr{H}_\mu:F^{p,\infty,\alpha}\to F^{p,\infty,\alpha}.$
	}

\end{rmk}

\begin{theorem}\label{compact4}
	Let $0<p,q\le\infty$ and $\mu$ be a positive Borel measure with  $\int_0^\infty\frac{d\mu(t)}{t}\ dt =1.$

Then $\mathscr{H}_\mu: F^{p,q,\alpha}\to F^{p,q,\alpha}$ is compact if and only if $\mu((0,1])=0$.
\end{theorem}
\begin{proof}
	Assume that $\mathscr{H}_\mu: F^{p,q,\alpha}\to F^{p,q,\alpha}$ is compact. Then $\mu(0,1)=0$ due to Theorem \ref{mainteo}.  Taking $f_n=\frac{u_n}{\|u_n\|_{q,\alpha}}$ and using that $\mathscr{H}_\mu(f_n)=\mu_n f_n$  then there exist an increasing sequence $(n_k)\subset {\mathbb N}$ and $g\in F^{p,q,\alpha}$ such that $\mu_{n_k}f_{n_k}\to g$ in $F^{p,q,\alpha}.$ Hence $\mu_{n_k}\to \|g\|_{p,q,\alpha}.$ Since
	$f_{n_k}(z)$ converges pointwise to $0$ and $(\mu_n)$ is decreasing one has that $\mu_n\to 0$ as $n\to \infty.$ Taking into account that $$\mu(\{1\})=\lim_{n\to \infty}\int_{[1,\infty)} \frac{d\mu(t)}{t^{n+1}}= \lim_{n\to\infty} \mu_n$$
	we obtain that $\mu(\{1\})=0.$ Therefore $\mu((0,1])=0$.
\par\medskip 	
Assume now that $\mu((0,1])=0$. Hence, from Theorem \ref{mainteo},  $\mathscr{H}_\mu(f)(z)=\int_1^\infty f(z/t)\frac{d\mu(t)}{t}$ and it is bounded on $F^{p,q,\alpha}$. Consider, for $\delta > 1,$
$$T_\delta(f)(z)=\int_\delta^\infty f\left(\frac{z}{t}\right) \frac{d\mu(t)}{t}.$$ We put $\beta = \frac{\alpha}{\delta^2} < \alpha$ and $A = \sqrt{\frac{2}{q\beta}}.$ Now for each $f\in F^{p,q,\alpha}$ and $t\ge \delta$ then
	$$
	\begin{array}{*2{>{\displaystyle}l}}
	\|D_{1/t} f\|^q_{p,q,\beta}& = q\beta\int_0^\infty M^q_p( f, \frac{r}{t})e^{ - q\frac{\beta}{2}r^2}r\ dr\\ & \\
	&\le q\beta\int_0^{tA}M^q_p( f, \frac{r}{t})e^{ - q\frac{\beta}{2}r^2}r\ dr + q\beta \int_A^\infty M^q_p(f,r)\sup_{t\ge \delta} t^2e^{ -q\frac{\beta t^2}{2}r^2}r\ dr
	\end{array}$$
	In the second integral we observe that $r\geq A\geq \frac{A}{\delta}$ implies that $\Psi(t):=t^2e^{-\frac{q\beta t^2}{2}r^2}$ decreases on $[\delta, \infty),$ hence
	$$
	\sup_{t\ge \delta} t^2e^{ -q\frac{\beta t^2}{2}r^2} = \delta^2 e^{-\frac{q\alpha}{2}r^2}.$$ On the other hand, the first integral is less than or equal to
	$$
	q\beta M_p(f,A)^q \int_0^\infty e^{-\frac{q\beta}{2}r^2}r\ dr \leq C\|f\|_{p,\infty,\alpha}^q$$ for some constant $C > 0,$ which only depends on $\alpha, \delta, q.$ Summarizing, we conclude that
	$$
	\|D_{1/t} f\|^q_{p,q,\beta} \leq C\|f\|_{p,\infty,\alpha}^q + \delta^2\|f\|_{p,q,\alpha}^q.$$ Since $F^{p,q,\alpha} \subset F^{p,\infty,\alpha}$ with continuous inclusion we conclude that $D_{1/t} f \in F^{p,q,\beta}$ and $\|D_{1/t} f\|_{p,q,\beta} \leq D \|f\|_{p,q,\alpha}$ for some constant $D$ independent on $t\geq \delta.$

Let $\gamma=\min\{p,q, 1\}$. Since $F^{p,q,\alpha}$ is $\gamma$-Banach space then
$$
\begin{array}{*2{>{\displaystyle}l}}
\|T_\delta(f)\|^\gamma_{p,q,\beta}& \le \int_\delta^\infty \|D_{1/t}f\|^\gamma_{p,q,\beta}\frac{d\mu(t)}{t}\\ & \\
&\lesssim \|f\|^\gamma_{p,q,\alpha}\int_\delta^\infty \frac{d\mu(t)}{t}.
\end{array}$$ Therefore $T_\delta: F^{p,q,\alpha}\to F^{p,q, \beta}$ is bounded and since the inclusion $F^{p,q,\beta}\subset F^{p,q,\alpha}$ is compact due to Lemma \ref{compactinclusion2} we obtain that also $T_\delta:F^{p,q,\alpha}\to F^{p,q,\alpha}$ is compact.

Finally, we recall that $\|D_{1/t}\|_{F^{p,q,\alpha}\to F^{p,q,\alpha}} = 1$ for $t > 1$ and use the following estimate:
$$\|\mathscr{H}_\mu(f)- T_\delta(f)\|^\gamma_{p,q, \alpha}\le \int_1^\delta \|D_{1/t}f\|^\gamma_{p,q, \alpha}\frac{d\mu(t)}{t}\le \|f\|^\gamma_{p,q,\alpha}\mu([1,\delta]).$$

Taking limits as $\delta\to 1$ we obtain the compactness of $\mathscr{H}_\mu$ from $F^{p,q,\alpha}$ into itself.
This completes the result.	
	
\end{proof}	

For the Fock spaces $F^p_\alpha$ we have the following result. It improves \cite[Theorem 1.2]{G-S}, were only the case $1 < p = q < \infty$ was considered.

\begin{theorem}\label{th:compact}
	Let $\mu$ be a Borel measure with  $\int_0^\infty\frac{d\mu(t)}{t}\ dt =1.$ The following are equivalent:
\begin{itemize}	
\item[(i)] There exists $0<p,q \le \infty$ such that $\mathscr{H}_\mu: F^p_\alpha\to F^q_\alpha$ is compact.
\item[(ii)] $\mu((0,1]) = 0.$	
\item[(iii)] $\mathscr{H}_\mu: F^p_\alpha\to F^p_\alpha$ is compact for all $0<p\le \infty.$
\end{itemize}
\end{theorem}
\begin{proof} By Theorem \ref{compact4} only (i) $\Longrightarrow$ (ii) needs a proof.
	Assume then that $\mathscr{H}_\mu: F^p_\alpha\to F^\infty_\alpha$ is compact. In particular $\mathscr{H}_\mu$ is bounded  and then $\mu(0,1) = 0.$ Let $B$ denote the closed unit ball in $F^p_\alpha.$ From compactness we obtain that the topology of $F^\infty_\alpha$ coincides with the compact-open topology on $\mathscr{H}_\mu(B).$ Now take $f_n\in B$ defined by
	$$
	f_n(z) = e^{nz - \frac{n^2}{2\alpha}},\ n\in {\mathbb N}.$$ Since
	$$
	\begin{array}{*2{>{\displaystyle}l}}
	\left|\mathscr{H}_\mu(f_n)(z)\right| & \leq  \int_1^\infty\left|e^{\frac{nz}{t}-\frac{n^2}{2\alpha}}\right|\frac{d\mu(t)}{t}\\ & \\ & \leq  e^{n|z|-\frac{n^2}{2\alpha}}\int_1^\infty\frac{d\mu(t)}{t}
	\end{array}$$ then $\mathscr{H}_\mu(f_n)$ goes to $0$ in the compact-open topology, hence in $F^\infty_\alpha,$ as $n\to \infty.$ In particular
	$$
	\lim_{n\to \infty}\left(\sup_{x>0}\left|\mathscr{H}_\mu(f_n)(x)\right|e^{-\frac{\alpha}{2}x^2}\right) = 0.$$ Since
	$$
	\begin{array}{*2{>{\displaystyle}l}}
	\mathscr{H}_\mu(f_n)(x) & = \int_1^\infty e^{\frac{nx}{t}-\frac{n^2}{2\alpha}}\frac{d\mu(t)}{t} = e^{nx-\frac{n^2}{2\alpha}}\cdot \mu\left(\{1\}\right) + \int_{(1,\infty)}e^{\frac{nx}{t}-\frac{n^2}{2\alpha}}\frac{d\mu(t)}{t}\\ & \\ & \geq e^{nx-\frac{n^2}{2\alpha}}\cdot \mu\left(\{1\}\right)\end{array}$$ then
	$$
	\lim_{n\to \infty}\left(\sup_{x>0}e^{nx-\frac{n^2}{2\alpha}}e^{-\frac{\alpha}{2}x^2}\right)\mu\left(\{1\}\right) = 0.$$ We consider the sequence $x_n = \frac{n}{\alpha}$ to finally conclude that $\mu\left(\{1\}\right) = 0.$
\end{proof}

\begin{corollary}  Let  $\mu$ be a positive Borel measure on $(0,\infty)$. 

(i) If $0<p<q\leq \infty$  and $\mathscr{H}_\mu:F^q_\alpha\to F^p_\alpha$ is bounded then $\mathscr{H}_\mu:F^q_\alpha\to F^q_\alpha$ is compact.

(ii) If $0<q<2$  and $\mathscr{H}_\mu:F^q_\alpha\to F^{2,q,\alpha}$ is bounded then $\mathscr{H}_\mu:F^q_\alpha\to F^q_\alpha$ is compact.

\end{corollary}
\begin{proof}
(i) Since $\mu_n \geq \mu(\{1\})$ then Remark \ref{rmk:5.1} gives $\mu_n\le C (n+1)^{(1/q-1/p)/2}$ and then $\mu\left(\{1\}\right) = 0.$ Hence $\mu\left((0,1]\right) = 0$ and we apply Theorem \ref{th:compact} to obtain the result.

(ii) It follows similarly because Theorem \ref{main2}  gives $\mu_n\le C (n+1)^{(q-2)4q}$. 
\end{proof}

\section{Compactness on weighted spaces}\label{sec:weighted}

We extend the characterization of compactness for Hausdorff operators acting on the Fock space $F^\infty_\alpha$ to more general Banach spaces of entire functions with weighted sup norms.
\par\medskip
Let $v$ be a continuous, positive and decreasing function on $[0, +\infty)$ such that
$$
\lim_{r\to +\infty}r^m v(r) = 0\ \ \forall m\in {\mathbb N}.$$ We extend $v$ to ${\mathbb C}$ by $v(z) = v(|z|)$ and consider the following weighted Banach spaces of entire functions
$$
H^\infty_v:= \left\{f\in H({\mathbb C}):\ \|f\|_v:=\sup_{z\in {\mathbb C}}|f(z)| v(z) < \infty\right\},$$
$$
H^0_v:= \left\{f\in H({\mathbb C}):\ |f| v \ \ \mbox{vanishes at infinity}\right\},$$ endowed with the norm $\|\cdot\|_v.$ For $v(r) = e^{-\frac{\alpha}{2}r^2}$ we obtain $H^\infty_v = F^\infty_\alpha.$
\par\medskip
Proceeding as in the proof of Theorem \ref{mainteo} one easily obtains the following result (see also \cite[Section 2]{Bonet}).
 \begin{proposition} \label{contHv} Let $\mu$ be a positive Borel measure on $(0,\infty)$. Then $\mathscr{H}_\mu: H^\infty_v\to H^\infty_v$ is a well defined and bounded operator if and only if $\mu(0,1) = 0$ and $\displaystyle\int_1^\infty\frac{d\mu(t)}{t} < \infty.$
 \end{proposition}
\par\medskip
The discussion that follows will later allow compactness to be related to weak compactness. We denote
$$
X:=\left\{F:H^\infty_v\to {\mathbb C}\ \mbox{linear}:\ F_{|B} \ \tau\mbox{-continuous}\right\},$$ where $B$ is the closed unit ball in $H^\infty_v$ and $\tau$ denotes the compact open topology. $X$ is endowed with the norm $\norm{F} = \sup_{f\in B}|F(f)|.$ According to \cite{Bierstedt},
$$
\Phi:H^\infty_v\to X^\ast,\ \ \Phi(f)(F):= F(f)$$ and
$$
R:X\to \left(H^0_v\right)^\ast,\ \ R(F):=F_{|H^0_v}$$ are surjective isometries. In particular $\left(H^0_v\right)^{\ast\ast}$ is isometric to $H^\infty_v.$

\begin{lem}\label{lem:bitransposed} Let us assume that $\mathscr{H}_\mu: H^\infty_v\to H^\infty_v$ is a bounded operator. Then $\mathscr{H}_\mu$ is the bi-transposed map of $\mathscr{H}_\mu:H^0_v\to H^0_v.$
\end{lem}
\begin{proof}
	From \cite[Prop. 2.1]{Bonet} we have that $\mathscr{H}_\mu:H^0_v\to H^0_v$ is also continuous. It suffices to check that $\mathscr{H}_\mu:H^\infty_v\to H^\infty_v$ is $\sigma(H^\infty_v, X)$-continuous, that is, $F\circ \mathscr{H}_\mu\in X$ for every $F\in X.$ This is a consequence of the fact that $\mathscr{H}_\mu:H^\infty_v\to H^\infty_v$ is continuous when $H^\infty_v$ is endowed with the $\tau$-topology. In fact, due to Proposition \ref{contHv}
	$$
	\left(\mathscr{H}_\mu f\right)(z) = \int_1^\infty f\left(\frac{z}{t}\right)\frac{d\mu(t)}{t}$$ and then
	$$M_\infty(\mathscr{H}_\mu(f),r)\leq  \int_1^\infty M_\infty(f,r/t)\frac{d\mu(t)}{t}\le M_\infty(f,r)\int_1^\infty\frac{d\mu(t)}{t}$$ for any $r>0$ and the proof is complete.
\end{proof}

\par\medskip
For the sake of completeness we recall a construction from \cite{Lusky}.
\par\medskip
For every $x > 0$  take $r_x$ a global maximum of $r\mapsto r^x v(r).$ Then $r_m\uparrow +\infty$ and $\norm{z^m}_v = r_m^m v(r_m)$ for every $m\in {\mathbb N}.$ For $0 < x < y$ and $f(z) = \sum_{k=0}^\infty b_k z^k,$ we put
$$
\left(V_{x,y}f\right)(z) = \sum_{k=0}^{[x]} b_k z^k + \sum_{[x] < k < [y]} \frac{[y]-k}{[y]-[x]} b_k z^k.$$ According to \cite[Prop. 5.2]{Lusky}, there are numbers $1 < x_1 < x_2 < \ldots$ diverging to infinity such that, for every $f\in H^\infty_v$ we have
$$
\norm{f}_v \asymp \sup_n \sup_{r_{x_{n-1}} \leq r \leq r_{x_{n+1}}}M_\infty(f_n,r)v(r),$$ where
$$
f_n = \big(V_{x_n,x_{n+1}} - V_{x_{n-1},x_n}\big) f.$$
\par\medskip
Let us now assume that $b_k = 0$ unless $k = [x_\ell]$ for some $\ell\in {\mathbb N}.$ Then
$$\big(V_{x_n,x_{n+1}} f\big)(z) = \sum_{k=0}^{[x_n]} b_k z^k.$$ Consequently
$$
f_n(z) = \sum_{[x_{n-1}] < k \leq [x_n]} b_k z^k.$$ That is, $f_n = 0$ in the case $[x_{n-1}] = [x_n]$ and $f_n = b_p z^p$ when $[x_{n-1}] < [x_n] = p.$ We now denote
$$
I:=\left\{[x_n]:\ n\in {\mathbb N}\right\},\ \ J:=\{n\in {\mathbb N}:\ [x_{n-1}] < [x_n]\}.$$ We have
$$
\norm{f}_v \asymp \sup_{n\in J}\sup_{r_{x_{n-1}} \leq r \leq r_{x_{n+1}}}M_\infty(f_n,r)v(r).$$ We observe that $[x_{n-1}] < [x_n] = p$ implies $x_{n-1}\leq [x_n]\leq x_{n+1},$ hence $r_{x_{n-1}}\leq r_p \leq r_{x_{n+1}}$ and
$$
\sup_{r_{x_{n-1}} \leq r \leq r_{x_{n+1}}}M_\infty(f_n,r)v(r) = |b_p|\cdot \norm{u_p}_v.$$ That is, if $f\in H^\infty_v,$ $f(z) = \sum_{k=0}^\infty b_k z^k,$ and $b_k = 0$ unless $k = [x_\ell]$ for some $\ell\in {\mathbb N}$ then
$$
\norm{f}_v \asymp \sup_{p\in I}|b_p|\cdot \norm{u_p}_v.$$

\begin{lem}\label{lem:iso} Let us assume that $\displaystyle\lim_{n\to \infty}\norm{u_n}_v^{\frac{1}{n}} = \infty$ and denote $w_n = \frac{u_n}{\norm{u_n}_v}.$ Then
	\begin{itemize}
		\item[(i)] $\Phi:\ell^\infty(I)\to H^\infty_v,\ \alpha = (\alpha_p)_{p\in I}\mapsto \sum_{p\in I} \alpha_p w_p(z)$ is an isomorphism into.
		\item[(ii)] $\alpha\in c_0(I)\Longleftrightarrow \Phi(\alpha)\in H^0_v.$
	\end{itemize}
\end{lem}
\begin{proof}
	(i) For any $\alpha = (\alpha_p)_{p\in I}\in \ell^\infty(I)$ we have that $f(z):=\displaystyle\sum_{p\in I} \alpha_p w_p(z)$ is an entire function. Moreover, from  \cite[Prop. 5.2]{Lusky} and the previous discussion $f\in H^\infty_v$ and $\norm{f}_v \asymp \sup_{p\in I}|\alpha_p|.$
	\par\medskip
	(ii) Let $\varphi$ denote the subspace of $\ell^\infty(I)$ consisting on those sequences with finitely many non-null coordinates. Then
	$$
	\Phi(c_0) = \Phi\big(\overline{\varphi}\big) = \overline{\Phi(\varphi)}^{\mbox{Range}(\Phi)} = \overline{\langle u_p:\ p\in I\rangle}^{\mbox{Range}(\Phi)} = \mbox{Range}(\Phi) \cap H^0_v.$$
\end{proof}

\begin{rmk}{\rm  For the weights $v(z) = \exp\left(-|z|^p\right),\ p\in {\mathbb N},$ Lemma \ref{lem:iso} should be compared with \cite[Theorem 3.2]{BG}.}
\end{rmk}

\begin{theorem}\label{th:weighted} Let us assume $\lim_{n\to \infty}\norm{u_n}_v^{\frac{1}{n}} = \infty$ \ and
	$\lim_{r\to \infty}\frac{v(tr)}{v(r)} = 0$ for every $t > 1$ and let $\mathscr{H}_\mu:H^\infty_v\to H^\infty_v$ be a bounded operator. The following conditions are equivalent:
	\begin{itemize}
		\item[(i)] $\mathscr{H}_\mu:H^\infty_v\to H^\infty_v$ is compact.
		\item[(ii)] $\mathscr{H}_\mu:H^0_v\to H^0_v$ is weakly compact.
		\item[(iii)] $\mu(\{1\}) = 0.$
	\end{itemize}	
\end{theorem}
\begin{proof} (i) $\Rightarrow $ (ii) follows trivially since $\mathscr{H}_\mu(H^0_v)\subset H^0_v$ and then $\mathscr{H}_\mu:H^0_v\to H^0_v$ is compact.
	\par\medskip
	(ii) $\Rightarrow $ (iii). From Lemma \ref{lem:bitransposed} we get $\mathscr{H}_\mu\big(H^\infty_v\big)\subset H^0_v.$ Take $f = \displaystyle\sum_{p\in I}\frac{u_p}{\norm{u_p}_v}.$
	According to Lemma \ref{lem:iso}, $f\in H^\infty_v.$ Since $\mathscr{H}_\mu(f)\in H^0_v$ and
	$$
	\left(\mathscr{H}_\mu f\right)(z) = \sum_{p\in I} \left(\int_1^\infty \frac{d\mu(t)}{t^{p+1}}\right)\cdot\frac{z^p}{\norm{u_p}_v},$$ we conclude from Lemma \ref{lem:iso} that $$\displaystyle\lim_{p\to \infty} \left(\int_1^\infty \frac{d\mu(t)}{t^{p+1}}\right) = 0.$$ That is, $\mu(\{1\}) = 0.$
	\par\medskip
	(iii) $\Rightarrow $(i). For any $\delta > 1$ we define
	$$
	T_\delta(f)(z) = \int_\delta^\infty f\left(\frac{z}{t}\right)\ \frac{d\mu(t)}{t}.$$ We put $w(z):= v\left(\delta^{-1}z\right).$ Then $\displaystyle\lim_{r\to \infty}\frac{v(r)}{w(r)} = 0,$ which implies that the inclusion map $H^\infty_w \to H^\infty_v$ is compact. As in the proof of Theorem \ref{compact4} we check that $T_\delta:H^\infty_v\to H^\infty_w$ is bounded and $T_\delta:H^\infty_v\to H^\infty_v$ is compact. As
	$$
	\norm{\mathscr{H}_\mu - T_\delta}_{H^\infty_v\to H^\infty_v}\leq \mu\left([1,\delta]\right)$$ goes to $0$ as $\delta\to 1,$ we are done.

\end{proof}

\section{Some examples}

Special examples can be provided using $d\mu(t)=\phi(t)dt$ for some positive measurable function $\phi$ with $\int_0^\infty\frac{\phi(t)}{t}dt<\infty$. Whenever
$d\mu(t)=\phi(t)dt$ and $d\eta(t)=\psi(t)dt$  we can consider $d\nu(t)=\phi*\psi(t)dt$ where $$\phi*\psi(t)=\int_0^\infty \phi(\frac{t}{s})\psi(s) \frac{ds}{s}.$$
It is easy to check that if $\mu$ and $\eta$ are positive Borel measures on $(0, \infty)$ of the form $d\mu = \phi(t)\chi_{(1,\infty)}(t)dt$ and $d\eta = \psi(t)\chi_{(1,\infty)}(t)dt$ then, formally, $\mathscr{H}_{\mu}\circ \mathscr{H}_{\eta}= \mathscr{H}_{\nu}$ that  is the Hausdorff operator with respect to $\nu$ on $(1,\infty).$ Moreover, since the monomials are eigenfunctions of any Hausdorff operator, it follows that
$\nu_n = \mu_n \eta_n$.

\begin{example} {\rm For every $a>0$ we consider the measure $d\mu(t) = \chi_{(1,\infty)}(t)\frac{dt}{t^a}.$ Then
$$
\mu_n = \int_1^\infty\frac{dt}{t^{n+a+1}} =\frac{1}{n+a}$$ and using Corollary \ref{th:from-inf-to-1} we obtain
$$
\mathscr{H}_\mu(F^\infty_\alpha) \subset \bigcap_{p>1}F^p_\alpha.$$ From Theorem \ref{case2} we get $\left(\mathscr{H}_\mu \circ\mathscr{H}_\mu \right)(F^\infty_\alpha) \subset F^1_\alpha.$
\par\medskip
In the case $a=1$ we recover the Hardy operator
$$
\left(\mathscr{H}_\mu f\right)(z) = \frac{1}{z}\int_0^z f(\xi)\ d\xi.$$ It is shown in \cite[Theorem 3.12]{bbf} that $\mathscr{H}_\mu \circ\mathscr{H}_\mu$ is compact when acting on weighted spaces of type $H^\infty_v.$ Theorems \ref{th:compact} and \ref{th:weighted} give the compactness of $\mathscr{H}_\mu$ on spaces $H^\infty_v$ and also when acting on $F^p_\alpha$ for every $\alpha >0$ and $0 < p\leq \infty.$
\par\medskip For $a = n\in {\mathbb N}$ we take $g(z) = z^n.$ Then $\mathscr{H}_\mu$ is related to Volterra operator
$$
\left(V_g f\right)(z) = \int_0^z f(\xi)g'(\xi)\ d\xi,$$ by $\left(V_g f\right)(z) = nz^n \left(\mathscr{H}_\mu f\right)(z).$
}
\end{example}

\begin{example} {\rm For every $a+1>b>0$ we consider the measure $$d\mu_{a,b}(t) = (t-1)^{b-1}t^{-a}\chi_{(1,\infty)}(t) dt.$$ Then
$$
\mu_n = \int_1^\infty(t-1)^{b-1}\frac{dt}{t^{n+a+1}} =\int_0^1 (1-s)^{b-1} s^{n+a-b}ds= B(b,n+a-b+1)$$
Using now that $B(p, n+\gamma)\asymp (n+1)^{-p}$  one can use Theorem \ref{case2} to obtain that $\mathscr{H}_{\mu_{a,b}}(F^\infty_\alpha) \subset F^1_\alpha$ for $b>1$ and Corollary \ref{th:from-inf-to-1} in the case $1/2 <b\le 1$ to obtain
$$
\mathscr{H}_{\mu_{a,b}}(F^\infty_\alpha) \subset \bigcap_{p>\frac{1}{b}}F^p_\alpha.$$ From Theorem \ref{case2} we get $\left(\mathscr{H}_{\mu_{a_1,b_1}} \circ\mathscr{H}_{\mu_{a_2,b_2}} \right)(F^\infty_\alpha) \subset F^1_\alpha$ whenever $b_1+b_2>1$.
\par\medskip
}
\end{example}

\begin{example}{\rm Let $\left(\lambda_k\right)_{k\in {\mathbb N}}$ and $\left(t_k\right)_{k\in {\mathbb N}}$ two sequences of positive numbers. We consider
		$$
		\mu = \sum_{k=1}^\infty \lambda_k \delta_{t_k}.$$ Then, formally,
$$
\left(\mathscr{H}_\mu f\right)(z) = \sum_{n=1}^\infty \frac{\lambda_k}{t_k}f\left(\frac{z}{t_k}\right).$$ From Theorem \ref{mainteo}, $\mathscr{H}_\mu$ is a bounded operator on $F^p_\alpha$ if and only if $\mu(0,1) = 0$ and $\int_1^\infty\frac{d\mu(t)}{t} < \infty,$ which means precisely $t_k\geq 1\ \forall k\in {\mathbb N}$ and $\displaystyle\sum_{k=1}^\infty\frac{\lambda_k}{t_k} < \infty.$

Also Theorem \ref{th:compact} gives that $\mathscr{H}_\mu$ is compact on $F^p_\alpha$ if and only if $t_k>1$ for all $k\in {\mathbb N}$ and $\displaystyle\sum_{k=1}^\infty\frac{\lambda_k}{t_k} < \infty.$

We have that $$
\mu_n = \int_1^\infty\frac{d\mu(t)}{t^{n+1}} = \sum_{k=1}^\infty\frac{\lambda_k}{t_k^{n+1}} $$

Let $A_0=\{k\in \mathbb N: t_k=1\}$. If $A_0\ne \emptyset$ then there exists $k_0\in A_0$ and  $\mu_n\ge \lambda_{k_0}$ for all $n$. Therefore  $\mathscr{H}_\mu$ is not compact and there is no $q<p$ such that $\mathscr{H}_\mu(F^p_\alpha)\subset F^q_\alpha$.

Assume now that $A_0=\emptyset$ and set $t_0=\inf_{k\in\mathbb N} t_k$.

If  $t_0>1$ then
$$\mu_n \le (\frac{1}{t_0})^n\sum_{k=1}^\infty\frac{\lambda_k}{t_k}$$ and
 hence Corollary \ref{th:from-inf-to-1} permits to conclude that $\mathscr{H}_\mu\left(F^\infty_\alpha\right)\subset F^1_ \alpha.$

In the case that $t_0=1$ and $\sum_{k=1}^\infty \frac{\lambda_k }{t_k-1} <\infty$ also we obtain $\mathscr{H}_\mu\left(F^\infty_\alpha\right)\subset F^1_ \alpha$
since $$\sum_{n=0}^\infty \mu_n= \sum_{k=1}^\infty\lambda_k(\sum_{n=0}^\infty \frac{1}{t_k^{n+1}})=\sum_{k=1}^\infty \frac{\lambda_k }{t_k-1} .$$
}
\end{example}

\section{Summing operators}
	
In this section we shall try to extend the results mentioned in the introduction about Hausdorff operators belonging to the  Schatten classes to the setting of summing operators to cover not only the case of Hilbert spaces but also to consider the action between different Fock spaces.

 Recall that for a Banach space $X$ we denote by $\ell^{weak}_p(X)$ and $\ell_p(X)$ the spaces of sequences $(x_n)\subset X$ such that $(\langle x^*,x_n\rangle)\in \ell_p$ for every $x^*\in X^*$ and $(\|x_n\|)\in \ell_p$ respectively. Also recall  that given two Banach spaces $X$ and $Y$ and  $1\le p\le q <\infty$, an operator  $T:X\to Y$ is said to be $(q,p)$-summing \cite{Diestel}, to be written $T\in \Pi_{(q,p)}(X,Y)$,  if $(T(x_n))\in \ell_q(Y)$ whenever $(x_n)\in \ell^{weak}_p(X)$. In the case $p=q$ they are called $p$-summing operators, denoted $T\in\Pi_{p}(X,Y)$ and $1$-summing operators are usually called absolutely summing.

 A special  class of $p$-summing operators are the so-called $p$-nuclear operators, to be denoted $N_p(X,Y),$ which  due to \cite[Proposition 5.23]{Diestel} can be described as those that can be written $T=\sum_{n=0}^\infty x_n^*\otimes y_n$ where $(x_n^*)\in \ell_p(X^*)$ and $(y_n)\in \ell^{weak}_{p'}(Y)$ where $1\le p <\infty$ and $1/p+1/p'=1$ and we use the notation  $x^*\otimes y$ for the rank one operator from $X$ to $Y$ given by $$x^*\otimes y(x)=\langle x^*,x\rangle y$$ for $x\in X,\ x^*\in X^*$ and $y\in Y$. The  $1$-nuclear operators are simply called nuclear and corresponds to $T=\sum_{n=0}^\infty x_n^*\otimes y_n$ for a bounded sequence $(y_n)\subset Y$ and $(x_n^*)\in \ell_1(X^*)$.

 In this section we include some results about $(q,p)$-summing  and $p$-nuclear operators on Fock spaces that are not Hilbert spaces. Several properties on such spaces will be useful in the sequel, for instance that for $1\le p<\infty$ the space $F^p_\alpha$, as subspace of $L^p_\alpha$, has  cotype $\tilde p=\max\{p,2\}$ (see for instance \cite{Diestel}). This implies that $\ell^{weak}_1(F^p_\alpha)\subset \ell_{\tilde p}(F^p_\alpha).$ This guarantees that if $\mathcal H_\mu(F^p_\alpha)\subset F^q_\alpha$ for $1\le p,q<\infty$ then $\mathcal H_\mu \in \Pi_{(\tilde p,1)}(F^p_\alpha,F^q_\alpha)$. In particular for $1\le p\le 2$,  if $\sup_{n\ge 0} \mu_n<\infty$ then we have that $\mathcal H_\mu \in \Pi_{(2,1)}(F^p_\alpha,F^p_\alpha)$.

 Moreover it is known that $F^p_\alpha$ is isomorphic to $\ell^p$ for $1\le p\le \infty$ (see \cite[Proposition 10]{GW} for $1<p<\infty$ and use the fact that $f^\infty_\alpha$ is isomorphic to $c_0$ \cite{galbis} for $p=1$). This allows to get a number of simple consequences.

Using Grothendieck's Theorem (see \cite[p.11]{Diestel}), since $F^2_\alpha\subset F^p_\alpha$ for $p\ge 2$, we obtain that if $\sup_{n\ge 0} \mu_n<\infty$ then   \begin{equation}\label{inclu0}\mathcal H_\mu \in \Pi_{1}(F^1_\alpha,F^p_\alpha), \quad p\ge 2. \end{equation}

Also using  general theory of $(q,p)$-summing operators (see \cite[Theorem 11.14]{Diestel}) acting on $L^\infty$-spaces we obtain that if $\mathcal H_\mu(F^\infty_\alpha)\subset F^p_\alpha$ for $1\le p<\infty$ then
    \begin{equation}\label{inclu1}
    \mathcal H_\mu \in \Pi_2(F^\infty_\alpha,F^p_\alpha), \quad 1\le p\le 2
    \end{equation}
    and
    \begin{equation}\label{inclu2}
    \mathcal H_\mu \in \Pi_{(q,p)}(F^\infty_\alpha,F^q_\alpha)=\Pi_{r}(F^\infty_\alpha,F^q_\alpha), \quad 1\le p<q<r<\infty, \quad q>2 .
    \end{equation}

    Note that in Corollary \ref{th:from-inf-to-1} we obtain that the inclusion $\mathcal H_\mu(F^\infty_\alpha)\subset F^p_\alpha$ for $1\le p\le 2$ gives $\sum_{n=0}^\infty \mu_n^p (n+1)^{-1/2}<\infty$.  We can improve this result  using (\ref{inclu1}).
    \begin{proposition} Let $\mu$ be a positive Borel measure defined on $(0,\infty)$.

    (i) If  $\mathcal H_\mu(F^\infty_\alpha)\subset F^p_\alpha$ for some $1\le p\le 2$ then $\sum_{n=0}^\infty \mu_n^2 (n+1)^{1/p-1/2}<\infty.$

    (ii) If  $\mathcal H_\mu(F^\infty_\alpha)\subset F^q_\alpha$ for some $q>2$ then $\sum_{n=0}^\infty \mu_n^q (n+1)^{1/2(1-q/2)}<\infty$.
    \end{proposition}
    \begin{proof} (i) Using (\ref{inclu1}) we have that, in particular, $\mathcal H_\mu:F^2_\alpha\to F^p_\alpha$ is $2$-summing. Since $\Big(\sqrt{\frac{\alpha^n}{n!}}u_n\Big) \in \ell_2^{weak}(F^2_\alpha)$ then
    $\Big(\mu_n \sqrt{\frac{\alpha^n}{n!}} \|u_n\|_{p,\alpha}\Big)\in \ell_2$. Therefore $$\sum_{n=0}^\infty \mu_n^2 (n+1)^{1/p-1/2}<\infty.$$

    (ii) We now use (\ref{inclu2}) for $p=2.$ The same argument as in (i) using that $\mathcal H_\mu:F^{2}_\alpha\to F^q_\alpha$ is $(q,2)$-summing gives  $\Big(\mu_n(n+1)^{1/2(1/q-1/2)}\Big)\in \ell_q$. This gives the result.
    \end{proof}

Let us now reformulate
(\ref{esti1}), (\ref{eqqq}) and (\ref{equa1}) in terms of sequences in $\ell^{weak}_p(X)$.
\begin{lem} \label{newlema} The following holds.

\begin{equation}\label{debil1}\Big(\frac{u_n}{\|u_n\|_{1,\alpha}}\Big) \in \ell^{weak}_1(f^\infty_\alpha),
\end{equation}
    \begin{equation}\label{debil12}\Big(\frac{u_n}{\|u_n\|_{p',\alpha}}\Big)\in \ell^{weak}_{p'}(F^{p}_\alpha), \quad 2\leq p < \infty,
\end{equation}
and
    \begin{equation}\label{debil2infty}\Big(\frac{u_n}{\|u_n\|_{p,\alpha}}\Big)\in \ell^{weak}_{p'}(F^{p}_\alpha), \quad 1< p \le 2.
\end{equation}
\end{lem} \begin{proof}
For every  $y^\ast\in (F^p_\alpha)^\ast$ (respectively $(f^\infty_\alpha)^\ast$) there is $g = \displaystyle\sum_{n=0}^\infty b_n u_n\in F^{p'}_\alpha$ (respectively $F^1_\alpha$) such that $\|y^\ast\|_{(F^p_\alpha)^\ast} = \|g\|_{p',\alpha}$ and $y^\ast(f) = \langle f, g\rangle_\alpha$ for every $f\in F^p_\alpha.$
Now observe that  $\langle f, u_n\rangle_{\alpha}= a_n \frac{n!}{\alpha^n}$ whenever $f(z)=\sum_{n=0}^\infty a_n z^n$, and
$$\|u_n\|_{p,\alpha}\|u_n\|_{p',\alpha}\asymp \frac{n!}{\alpha^n}, \quad 1\le p \le \infty.$$

Assume first $y^\ast\in (f^\infty_\alpha)^\ast$. Then
$$
\left|y^\ast\Big(\frac{u_n}{\|u_n\|_{1,\alpha}}\Big)\right|  \asymp \frac{|b_n|}{\|u_n\|_{1,\alpha}}\frac{ n!}{\alpha^n}\asymp |b_n|\|u_n\|_{\infty,\alpha} \asymp |b_n|\|u_n\|_{1,\alpha}(n+1)^{-1/2}.
$$ Hence using (\ref{esti1}) we conclude (\ref{debil1}).

Assume now that $y^\ast\in (F^p_\alpha)^\ast$ for $2\leq p<\infty.$ Then
$$
\left|y^\ast\Big(\frac{u_n}{\|u_n\|_{p',\alpha}}\Big)\right|  \asymp \frac{|b_n|}{\|u_n\|_{p',\alpha}}\frac{ n!}{\alpha^n}\asymp |b_n|\|u_n\|_{p,\alpha} \asymp |b_n|\sqrt{\frac{n!}{\alpha^n}}(n+1)^{1/2(1/2-1/p')}.
$$ Hence, since $1<p'\le 2$, using (\ref{equa1}) we conclude (\ref{debil12}).

(\ref{debil2infty}) follows similarly using now (\ref{eqqq}) since $p'\ge 2$.

\end{proof}

\begin{proposition}\label{proppq} Let $1 < p \le \infty$ and $p'\leq q$
and  $\mathscr{H}_\mu:F^{p}_\alpha\to F^{p}_\alpha$ is $(q,p')$-summing. Then
$$
\sum_{n=0}^\infty \mu_n^{q} (n+1)^{q(\frac{1}{\max\{p,2\}}-\frac{1}{2})} <\infty.$$

\end{proposition}
\begin{proof}
Assume $2<p\le \infty$.
 Using (\ref{debil12}) or (\ref{debil1}) in Lemma \ref{newlema} we obtain $\Big(\mathscr{H}_\mu( \frac{u_n}{\|u_n\|_{p',\alpha}})\Big)\in \ell_{q}(F^{p}_\alpha)$. Now  use that $ \frac{\|u_n\|_{p,\alpha}}{\|u_n\|_{p',\alpha}}\asymp (n+1)^{\frac{1}{p}-\frac{1}{2}}$ to complete the implication.

The case $1<p\le2$ follows similarly using  (\ref{debil2infty}).
\end{proof}

\par\medskip
\begin{theorem} \label{inftyuno} Let $\mu$ be a positive Borel measure defined on $(0,\infty)$.

$\mathscr{H}_\mu:f^\infty_\alpha\to F^1_\alpha$ is absolutely summing if and only if
$\sum_{n=0}^\infty \mu_n<\infty$.
\end{theorem}
\begin{proof} Assume that $\mathscr{H}_\mu:f^\infty_\alpha\to F^1_\alpha$ is absolutely summing.
Let $(y_n)=(\frac{u_n}{\|u_n\|_{1,\alpha}})\subset f^\infty_\alpha$. Hence
 using    (\ref{debil1}) and the assumption we conclude that $$\left(\mathscr{H}_\mu(y_n)\right)= \left(\mu_n\frac{u_n}{\|u_n\|_{1,\alpha}}\right)\in \ell_1(F^1_\alpha).$$
This gives the direct implication.

Assume now $\sum_{n=0}^\infty \mu_n<\infty$ and let $(f_k)\in \ell_1^{weak}(f^\infty_\alpha)$ be given.
We know that $\mathscr H_\mu:f^\infty_\alpha\to F^1_\alpha$ is bounded and $\mathscr H_\mu(f)(z)=\int_1^\infty f(\frac{z}{t})\frac{d\mu(t)}{t}$. We write $f(z)=\langle f,  \tilde K_z\rangle_\alpha e^{\frac{\alpha}{2}|z|^2}$ where $\tilde K_z(w)=\frac{K_\alpha(z,w)}{\|K_\alpha(z,\cdot)\|_{1,\alpha}}$, which allows to write
\begin{eqnarray*}
\sum_{k=1}^m \|\mathscr H_\mu(f_k)\|_{1,\alpha}
&\le& C\sum_{k=1}^m \int_1^\infty(\int_\C|f_k(\frac{z}{t})|e^{-\frac{\alpha}{2}|z|^2} dA(z))\frac{d\mu(t)}{t}\\
&=& C\int_1^\infty\int_\C\Big(\sum_{k=1}^m|\langle f_k, \tilde K_{z/t}\rangle_\alpha  |\Big)e^{\frac{\alpha}{2t^2}|z|^2} e^{-\frac{\alpha}{2}|z|^2}dA(z) \frac{d\mu(t)}{t}\\
&\le& C\|(f_k)\|_{\ell^{weak}_1(f^\infty_\alpha)} \int_1^\infty\int_\C e^{-\frac{\alpha}{2}|z|^2(1-\frac{1}{t^2})}dA(z) \frac{d\mu(t)}{t}\\
&\le& C\|(f_k)\|_{\ell^{weak}_1(f^\infty_\alpha)} \int_1^\infty(1-\frac{1}{t^2})^{-1} \frac{d\mu(t)}{t}\\
&\le& C\|(f_k)\|_{\ell^{weak}_1(f^\infty_\alpha)} \sum_{n=0}^\infty \mu_{2n}.
\end{eqnarray*}
The proof is now complete.
\end{proof}

\begin{theorem} \label{nuclear} Let $\mu$ be a positive Borel measure defined on $(0,\infty)$. Then each of the following statements implies the one that follows it.

(i) $\sum_{n=0}^\infty \mu_n<\infty$.

(ii) $\mathscr{H}_\mu:f^\infty_\alpha\to f^\infty_\alpha$ is nuclear.

(iii) $\mathscr{H}_\mu:f^\infty_\alpha\to f^\infty_\alpha$ is absolutely summing.

(iv) $\sum_{n=0}^\infty \mu_n (n+1)^{-1/2}<\infty$.

\end{theorem}
\begin{proof}
(i) $\Longrightarrow$ (ii). We write $\mathscr{H}_\mu =\sum_{n=0}^\infty x_n^*\otimes y_n$ where
\begin{equation}
 x_n^*= \mu_n\|u_n\|_{\infty,\alpha} \frac{\alpha^n}{n!} u_n, \qquad  y_n=\frac{u_n}{\|u_n\|_{\infty,\alpha}}.
\end{equation}
The series converges absolutely, using that $\sum_{n=0}^\infty \mu_n<\infty$ because $$\|x_n^*\otimes y_n\|_{f^\infty_\alpha\to f^\infty_\alpha}=\|x_n^*\|_{1,\alpha}\|y_n\|_{\infty,\alpha}\asymp \mu_n.$$
Using now that $\|y_n\|_{\infty,\alpha}=1$ and  $\|x_n^*\|_{1,\alpha}= \mu_n\|u_n\|_{\infty,\alpha} \frac{\alpha^n}{n!} \|u_n\|_{1,\alpha}\asymp \mu_n$ we obtain that $\mathscr H_\mu$ is nuclear.

(ii) $\Longrightarrow$ (iii) This is well known but we include the easy proof for completeness. Assume   $$\mathscr{H}_\mu =\sum_{n=0}^\infty F_n\otimes G_n$$ with $(F_n)\in \ell_1(F^1_\alpha)$ and $\sup \|G_n\|_{\infty,\alpha}<\infty.$ Let $(f_k)\in \ell^{weak}_1(f^\infty_\alpha)$ be given. Using that $\ell_1(X)\subset (\ell_\infty(X^*))^*$ isometrically we can write
 that \begin{eqnarray*}\sum_{k=1}^\infty\|\mathscr{H}_\mu(f_k)\|_{\infty,\alpha}&=& \sup_{\|g_k\|_{1,\alpha}\leq 1}\Big|\sum_{k=1}^\infty \langle\mathscr{H}_\mu(f_k),g_k\rangle_\alpha\Big|\\
&= & \sup_{\|g_k\|_{1,\alpha}\leq 1}\Big|\sum_{k=1}^\infty \sum_{n=0}^\infty \langle F_n,f_k \rangle_\alpha \langle G_n,g_k\rangle_\alpha\Big|\\
&\le & \sup_{\|g_k\|_{1,\alpha}\leq 1}\sum_{n=0}^\infty\Big|\langle F_n,\sum_{k=1}^\infty \langle G_n,g_k\rangle_\alpha f_k\rangle_\alpha\Big|\\
&\le & \sup_{\|g_k\|_{1,\alpha}\leq 1}\sum_{n=0}^\infty \|F_n\|_{1,\alpha} \|\sum_{k=1}^\infty \langle G_n,g_k\rangle f_k\|_{\infty,\alpha}\\
&\lesssim & \|(F_n)|_{\ell_1(F^1_\alpha)} \|(f_k)\|_{\ell^{weak}_1(f^\infty_\alpha)}
 \end{eqnarray*}
 where the last inequality follows using that $\|(x_n)\|_{\ell^{weak}_1(X)}= \sup_{|\lambda_n|\leq1} \| \sum_{n=0}^\infty \lambda_n x_n\|_X$ and $|\langle G_n, g_k\rangle_\alpha| \leq \|G_n\|_{\infty,\alpha}\lesssim 1.$

(iii) $\Longrightarrow$ (iv) Select as in Theorem \ref{inftyuno} $y_n=\frac{u_n}{\|u_n\|_{1,\alpha}}$. Hence $$\sum_{n=0}^\infty\|\Big(\mathcal H(y_n)\Big)\|_{\infty,\alpha}= \sum_{n=0}^\infty\mu_n\frac{\|u_n\|_{\infty,\alpha}}{\|u_n\|_{1,\alpha}}\asymp \sum_{n=0}^\infty\mu_n (n+1)^{-1/2}<\infty.$$

\end{proof}

Let us now give some conditions connecting with $p$-nuclearity.
\begin{proposition} \label{pro1} Let $\mu$ be a Borel positive measure in $(0,\infty)$.

  (i) If  $\displaystyle\sum_{n=0}^\infty \mu_n < \infty$ then $\mathscr{H}_\mu:F^q_\alpha\to F^q_\alpha$ is nuclear for any $1\le q\le \infty$.

  (ii) If $1 < p\le q \le 2$ and  $\displaystyle\sum_{n=0}^\infty \mu_n^p < \infty$ then $\mathscr{H}_\mu:F^q_\alpha\to F^q_\alpha$ is $p$-nuclear.

  (iii) If $\max\{p,2\}\le q$  and $\displaystyle\sum_{n=0}^\infty \mu_n^p(n+1)^{p(\frac{1}{2}-\frac{1}{q})}<\infty$ then $\mathscr{H}_\mu:F^q_\alpha\to F^q_\alpha$ is $p$-nuclear.

\end{proposition}
\begin{proof}  Under the hypothesis of (i), (ii) or (iii), it follows from Theorem \ref{mainteo} that  $\mathscr{H}_\mu$ is a bounded operator on $F^q_\alpha$ and $$\left(\mathscr{H}_\mu f\right)(z) = \sum_{n=0}^\infty \mu_n \frac{\alpha^n}{n!} \langle f, u_n\rangle_\alpha u_n(z)=\sum_{n=0}^\infty x_n^*\otimes y_n(f)$$ for every $f = \displaystyle\sum_{n=0}^\infty a_n u_n\in F^q_\alpha,$ where
	we put $y_n:= \frac{u_n}{\|u_n\|_{q,\alpha}}\in F^q_\alpha$ and  $x_n^\ast\in (F^q_\alpha)^\ast$ defined by
$$
x_n^\ast(f):= \mu_n \frac{\alpha^n}{n!} \|u_n\|_{q,\alpha} \langle f, u_n\rangle_\alpha.$$ Then $\|y_n\|_{q,\alpha} = 1$ and \begin{equation}\label{normaxn}\|x_n^\ast\|_{(F^q_\alpha)^\ast} = \mu_n \frac{\alpha^n}{n!} \|u_n\|_{q,\alpha}\|u_n\|_{q',\alpha} \asymp \mu_n.
\end{equation} 	

(i)  The condition $\displaystyle\sum_{n=0}^\infty \mu_n < \infty $ gives that the series $\sum_{n=0}^\infty x_n^*(f)y_n$ converges absolutely in $F^q_\alpha$ for every $1\le q\le \infty.$ Hence $\mathscr{H}_\mu f = \sum_{n=0}^\infty x_n^\ast(f) y_n$ with convergence in $F^q_\alpha$ and $\mathscr{H}_\mu$ is nuclear. We observe that in the case $q = \infty$ we even conclude that $\mathscr{H}_\mu(F^\infty_\alpha)\subset f^\infty_\alpha.$

(ii) For $1<q\leq 2$ we have that $(y_n)$ is a Schauder basis in $F^q_\alpha$ (see \cite[Proposition 1.7]{GW}) and we can put
$$ \mathscr{H}_\mu (f) = \sum_{n=0}^\infty x_n^\ast(f) y_n,$$ where the series converges in $F^q_\alpha.$ Moreover, from (\ref{normaxn}) we have that $(x^*_n)\in \ell_p((F^q_\alpha)^\ast).$ Therefore, in order to see that $\mathscr{H}_\mu$ is $p$-nuclear it suffices to check that $(y_n)\in \ell^{weak}_{p'}(F^{q}_\alpha)$.
From (\ref{debil2infty}) in Lemma \ref{newlema} $(y_n)\in \ell^{weak}_{q'}(F^q_\alpha)$. Since $\ell^{weak}_{q'}(F^q_\alpha)\subset \ell^{weak}_{p'}(F^q_\alpha)$ because $p\le q$, the proof of (ii) is complete.

(iii) We write $\mathscr{H}_\mu =\sum_{n=0}^\infty (\mu_n \|u_n\|_{q',\alpha}\frac {\alpha^n}{n!}u_n)\otimes (\frac{u_n} {\|u_n\|_{q',\alpha}}).$ From (\ref{debil12}) Lemma \ref{newlema}, since $q\ge 2$ we have that  $(\frac{u_n} {\|u_n\|_{q',\alpha}})\in \ell^{weak}_{q'}(F^q_\alpha)\subset \ell^{weak}_{p'}(F^q_\alpha)$ using that $p\le q$. Now, using  $$\|u_n\|_{q',\alpha}^2\frac{\alpha^n}{n!}\asymp (n+1)^{\frac{1}{2}-\frac{1}{q}} $$ we obtain $\Big(\mu_n \|u_n\|_{q',\alpha}\frac {\alpha^n}{n!}u_n\Big)\in \ell^p(F^{q'}_\alpha)$ under the assumption
$ \sum_{n=0}^\infty \mu_n^p(n+1)^{p(\frac{1}{2}-\frac{1}{q})}<\infty$ and the result is shown.
\end{proof}

\begin{proposition}\label{prop:9.3} Let $1 \leq p \leq 2.$

(i) If
$
\sum_{n=0}^\infty \mu_n^{p} (n+1)^{\frac{2-p}{2}} <\infty$ then $\mathscr{H}_\mu:F^{p'}_\alpha\to F^{p'}_\alpha$ is $p$-summing.

(ii) If  $\mathscr{H}_\mu:F^{p'}_\alpha\to F^{p'}_\alpha$ is $p$-summing then
$
\sum_{n=0}^\infty \mu_n^{p} (n+1)^{\frac{p-2}{2}} <\infty.$

(iii) If  $\mathscr{H}_\mu:F^{p}_\alpha\to F^{p}_\alpha$ is $p'$-summing then $(\mu_n)\in \ell_{p'}.$

\end{proposition}
\begin{proof}
(i) follows from (iii) in Proposition \ref{pro1} for $q=p'$, using that $p$-nuclear implies $p$-summing (\cite{Cohen}).

(ii) follows from Proposition \ref{proppq}.

(iii) The case $p=1$ means that $\mathscr{H}_\mu$ bounded on $F^{1}_\alpha$ implies $(\mu_n)$ is bounded. The case $1<p\le 2$ follows from the main result in \cite{jkmr} (see also \cite[p.249]{woj}) or using Proposition \ref{proppq}.
\end{proof}

\end{document}